\documentclass[12pt,a4paper]{article} 
\usepackage{amsmath,amssymb,amsthm,amsfonts,amscd,euscript,verbatim, t1enc, newlfont}
\usepackage{bbm}
\usepackage{hyperref}
\usepackage{graphicx}
\usepackage[all]{xy}
\usepackage{color}
\usepackage{pgfplots}
\usepackage{tikz-cd}
\pgfplotsset{width=7cm, compat=1.10}
\usepgfplotslibrary{fillbetween}
\theoremstyle{definition}
\newtheorem{theo}{Theorem}[subsection]

\newtheorem{pr}[theo]{Proposition}

 \newtheorem{lem}[theo]{Lemma}
 \newtheorem{coro}[theo]{Corollary}
  
\theoremstyle{remark}
\newtheorem{rema}[theo]{Remark}

\theoremstyle{definition}
\newtheorem{defi}[theo]{Definition}

\newtheorem{defnc}[theo]{Construction-Definition}

\numberwithin{equation}{subsection}

\newcommand\cu{\underline{C}}
\newcommand\du{\underline{D}}

\newcommand\au{\underline{A}}
\newcommand\bu{\underline{B}}

\newcommand\chow{\operatorname{Chow}}

\newcommand\smo{\operatorname{Smooth}}

\newcommand\q{{\mathbb{Q}}}

\newcommand\obj{\operatorname{Obj}}
\newcommand\hw{{\underline{Hw}}}

\newcommand\ood{\mathbbm{1}_{\mathfrak{D}}} %\mathbb
\newcommand\md{\mathfrak{D}} %
\newcommand\mmd{\mathcal{M}_{\mathfrak{D}}}

\newcommand\ab{{Ab}}
\DeclareMathOperator\cha{\operatorname{char}}

\newcommand\zop{{\mathbb{Z}[\frac{1}{p}]}}

\newcommand\z{{\mathbb{Z}}}

\newcommand\wcho{w_{\chow}}
\newcommand\wchosw{w_{\smo}{}}
\newcommand\wchose{w_{\smo}^{eff}}

\newcommand\hrt{{\underline{Ht}}}

\newcommand\ns{\{0\}}

\newcommand\spe{\operatorname{Spec}}

\newcommand\p{\mathbb{P}}
\newcommand\id{\operatorname{id}}
 \newcommand\lan{\langle}
\newcommand\ra{\rangle}

\newcommand\sv{\operatorname{SmVar}}
\newcommand\spv{\operatorname{SmPrVar}}

\DeclareMathOperator\kar{\operatorname{Kar}}

\DeclareMathOperator\imm{\operatorname{Im}}
\DeclareMathOperator\co{\operatorname{Cone}}

\newcommand\mgl{\operatorname{MGl}}
\newcommand\modd{\operatorname{Mod}}

\newcommand\dmerb{DM^{eff}_{R}}

\newcommand\dmgep%{DM^{eff}_{gm,\zop}}  %
{{DM^{\,\, {eff}}_{\scalebox{0.6}{gm},\zop}}{}}

\newcommand\dmgepr %{DM^{eff,R}_{gm,\zop}} %
{{DM^{\,\, {eff}}_{\scalebox{0.5}{gm},\zop,R}}{}}
\newcommand\dmgepq %{DM^{eff,\q}_{gm,\zop}} %
{{DM^{\,\, {eff}}_{\scalebox{0.5}{gm},\zop,\q}}{}}

\newcommand\dmr{DM_R}

\newcommand\sht{SH(k)}
\newcommand\afo{\mathbb{A}^1}

\begin{document}
\title{Smooth weight structures and birationality filtrations on motivic categories }
\author{M.V. Bondarko, D.Z. Kumallagov
   \thanks{ %%!!!!
 %The main results of the paper were  obtained under support of 
  Section 1, 2, and 5 were supported by the RFBR grant no. 19-31-90074. Sections 3--4 of the paper were supported by the Russian Science Foundation grant no. 20-41-04401.}} 
 %16-11-00200.
\maketitle
\begin{abstract} %In this paper we 
 We study various triangulated motivic categories and introduce a vast family of aisles (these are certain classes of objects) in them. These aisles are defined in terms of the corresponding "motives" (or motivic spectra) of smooth varieties in them; we  relate them to the corresponding homotopy $t-$structures.
 %We prove that these aisles %(as well as the corresponding 
 %(along with the filtrations they give) % generated by them) 
  %widely generalize %Chow weight structures and
   %slice filtrations. Further, we 
   We describe our aisles in terms of %Zariski
  stalks at function fields and prove that they widely generalize the ones corresponding to slice filtrations. %It follows that
   Further, the filtrations on the %corresponding 
   "homotopy hearts" $\hrt_{hom}^{eff}$ of the corresponding effective subcategories that are induced by these aisles can be described in terms of (Nisnevich) sheaf cohomology as well as in terms of the Voevodsky contractions $-_{-1}$.
%This allows us to
  Respectively, we  express the condition for an object of $\hrt_{hom}^{eff}$ to be weakly birational (i.e., that its $n+1$th contraction is trivial or, equivalently, the Nisnevich cohomology vanishes in degrees $>n$ for some $n\ge 0$) in terms of these aisles; this statement generalizes %several 
well-known results of Kahn and Sujatha. %Note that a similar definition was introduced by Pelaez, and it is closely related to the slice filtration. 
Next, these classes define weight structures $w_{Smooth}^{s}$ (where $s=(s_{j})$ %is the non-decreasing sequence 
 are non-decreasing sequences  parameterizing our %weight structure)
 aisles) that vastly generalize the Chow weight structures $w_{\chow}$ defined earlier.
Using %the general machinery of weight structures
 general abstract nonsense we also construct the corresponding {\it adjacent} $t-$structures $t_{Smooth}^{s}$ %, %which leads us to the 
  %that %induce
   %yield an interesting  %(increasing) 
   and prove  that they give the birationality filtrations on $\hrt^{eff}_{hom}$. %The elements of this filtration are weakly birational objects.

  %We also obtain the corresponding weight and $t-$structures 
   Moreover, some of these weight structures induce weight structures on the corresponding $n-$birational motivic categories (these are the localizations by the levels of the slice filtrations). Our results %obtained also enables
    also %allow us to prove some results on 
     yield some new unramified cohomology calculations.  %, vastly 
    %generalizing the corresponding results of Kahn and Sujatha's. % for objects from $Ht_{hom}^{\mathfrak{D}}$ for various ''motivic'' categories $\mathfrak{D}$. 
     %Also note that our approach does not    require any compactifications, providing their ''motivic analogues'', and is applicable for an arbitrary coefficient ring $R\subset \q$ for our motivic categories.
  %These statements can also be applied to the several versions of Chow-weight homologies, defined at all these $\mathfrak{D}$.
\end{abstract}

\tableofcontents

\section{Introduction}\label{intro}

In %the paper
  \cite{bokum} the authors constructed and studied in detail the {\it Chow weight structure} $\wcho$ on the category $\dmerb(k)$ of Voevodsky's motives (and on its ''stable'' version $\dmr(k)$), where $k$ is a perfect field and $R$ is the coefficient ring.
%Note that, in contrast to the previously defined weight structures (see \cite{bws}, \cite{bzp}).
 The main advantage of this definition (in contrast to the earlier one in \cite{bzp}) was that it %is independent
 did not depend  on %various 
 any resolution of singularities results, and the characteristic $p$ of $k$ was not required to be invertible in $R$ (if  $p>0$).

In the current text we consider similar (but much more  general) definitions %(but much more general, formulated in terms of aisles)  
 in various  motivic categories, and %compare it
 relate the corresponding classes of objects (that are {\it aisles}) to the corresponding (well-known) homotopy $t-$structures $t_{hom}$. %For this reason we introduce and discuss these $t_{hom}-$structures in a rather %thoroughful and 
 %axiomatic  manner; see \S \ref{thomaxc}. Next we show that t
 The filtrations given by these aisles generalize the slice and (Chow) weight ones (the latter are considered in several papers of the first author). %Also, we 
  We also consider the effective versions of our motivic categories and obtain several conditions for an object from the heart of the homotopy $t-$structure $\hrt^{eff}_{hom}$ to be {\it weakly birational} (see Proposition \ref{epowbo} and Remark \ref{eot}(1)). %, and thus generalize the corresponding results of $\cite{kabir}$.
Further, %due to the fact that
 since our %new 
  weight structures $w_{Smooth}^{s}$ (generated by $\mathcal{M}(\sv)\lan j \ra[s_{j}]$, where $s=s_{j}$ is a non-decreasing sequence) are {\it smashing}, %for it 
 there exist  $t$-structures   $t_{Smooth}^{s}$ %that is
   (right) {\it adjacent} to $w_{Smooth}^{s}$. This allows us to prove interesting properties of our filtration on $\underline{Ht}_{hom}^{eff}$: the levels  %The  
 $F^{i}$ of this filtration consist of $i-$birational objects, and they give right adjoint functors to the %corresponding 
  embeddings $\hrt^{i-bir} \hookrightarrow \hrt^{eff}_{hom}$. Thus we  generalize the corresponding results of \cite{kabir}.
We also obtain a curious statement about unramified cohomology (see Proposition \ref{spobc}).

%Next, we ''induce'' our $w_{\smo}$ and $t_{\smo}$ on categories of $n$-birational motives. In particular, this gives interesting statements about the behavior of $t_{\smo}$ and $t_{hom}$ in relation to each other in $n$-birational setting.

Let us now describe the contents of the paper. Some more information of this sort can be found at the beginnings of sections. By $\mathfrak{D}^{eff} \subset \mathfrak{D}$ we denote the %considered
  motivic categories we consider. 

In \S\ref{sprel} we recall several definitions and results on triangulated categories, $t-$structures, and recall the %
%relevant for us 
motivic categories relevant to us (these are $SH^{S^{1}}(k)$, $SH^{eff}(k) \subset SH(k)$, $DM^{eff}(k) \subset DM(k)$, and $D_{\mathbb{A}^{1}}^{eff}(k)\subset D_{\mathbb{A}^{1}}(k)$)  
along with their basic properties. %We also recall some of the functors between our categories.
Further we discuss homotopy $t-$structures on them in a rather axiomatic manner. Some of the details are postponed until \S\ref{sother}. %, and recall some of their properties. %prove that they possesses certain properties. %the fulfillment of the required axioms in the case of $SH^{eff}(k) \subset SH(k)$. 

In \S\ref{coresul} we define %Chow 
 "smoothly generated" aisles corresponding to non-decreasing sequences $s=(s_j)$, in our motivic categories, 
and obtain several comparison %??????!
theorems. Next, we describe these %orthogonals 
 aisles in  terms of   stalks corresponding to function fields. Further, we %construct some increasing
consider the corresponding %Chow-weight??
  smooth filtrations on $\hrt^{eff}_{hom}$ and prove some of their properties. %, closely related to weak birationality. 
Finally, using %the developed 
  our results we prove that the homotopy $t-$structure  restricts to $\mathfrak{D}^{n-bir}$.
%?????In this section we illustrate our general results by the case of the stable motivic homotopy category $SH(k)$ only. 

In \S\ref{weismot} we recall %the
  basic definitions and properties of weight structures. % (predominantly on compactly generated categories). 
 Next we define the main weight structures on our categories --- the {\it $\mathcal{A}_{s}$-smooth} ones,
and relate them to unramified cohomology. Using %our 
general existence results, we also define {\it $\mathcal{A}_{s}$-smooth} t-structures (that are right adjacent to our $w_{Smooth}^{eff, s}$), and %use them to construct some $i-$birational filtrations on 
 relate them to the birationality filtration on $\hrt^{eff}_{hom}$. %We also 
  Next we relate our weight structures to the corresponding $n-$birational motivic categories. In particular, we %obtain some conditions for 
 study the weight-exactness of the functors $-\lan n \ra$ and of the localizations $p_{n}: \mathfrak{D}^{eff} \to \mathfrak{D}^{n-bir}$. % and stabilization.

In \S\ref{sother} we discuss (in more detail) some definitions and results  for motivic categories different from $SH^{eff}(k) \subset SH(k)$.
% categories $SH^{S^{1}}(k)$, $DM^{eff}(k) \subset DM(k)$, and $D_{\mathbb{A}^{1}}^{eff}(k)\subset D_{\mathbb{A}^{1}}(k)$. One can also 
 We explain that %obtain more examples by
  localizing coefficients for these categories yields new examples.  This gives $R$-linear versions of the categories above, where $R\subset \q$; it follows that  the category  $SH(k)^{+}$ can also be added to the examples mentioned above.
%For this purpose we recall the theory of localizing coefficients for general compactly generated motivic categories; this yields $R$-linear version of the categories above, where $R\subset \q$.

%In Appendix \ref{olctc} we collect some notations and facts concerning the localization of a triangulated category, which are necessary to study the $R-$linear versions of $\mathfrak{D}_{R}^{eff} \subset \mathfrak{D}_{R}$.

\section{Preliminaries}\label{sprel}

In \S\ref{snotata} we give some definitions and conventions related to (mostly) triangulated categories.\\
In \S\ref{wat} we recall basic definitions and properties of t-structures.\\
In \S\ref{vgabc} we %briefly recall the
 recall some basics on various motivic categories.\\ % and functors between them.\\
In \S\ref{thomaxc} we introduce %some axioms on 
  and discuss homotopy $t$-structures on these categories.\\
In \S\ref{parcasosh} we % illustrate 
 discuss our axioms mainly in the case of the categories $SH^{eff}(k) \subset SH(k)$.

\subsection{Categorical definitions and notation}\label{snotata}
\begin{itemize}
\item Given a category $\bu$ and $M, N \in \obj \bu$, we say that $M$ is a {\it retract} of $N$ if $id_{M}$ can be factored through $N$ (recall that if $\bu$ is triangulated
then $M$ is a retract of $N$ if and only if $M$ is its direct summand).

\item A subcategory $\du$ of $\bu$ is said to be {\it retraction-closed} in $\bu$ if it contains all $\bu$-retracts of its objects.

\item The full subcategory $\kar_{\bu}(\du)$ of $\bu$ whose objects are all $\bu$-retracts of objects of $\du$ will be called the {\it retraction-closure} of $\du$ in $\bu$. It is easily
seen that $\kar_{\bu}(\du)$ is retraction-closed in $\bu$; if $\bu$ and $\du$ are additive then $\kar_{\bu}(\du)$ is additive as well.

\item We say that an additive category $\du$ is {\it Karoubian} if any its idempotent endomorphism is isomorphic to the composition of a retraction and a coretraction of the type $M\bigoplus N \to M \to M\bigoplus N$.

%\item The symbol $\cu$ below will always denote some triangulated category. For a given class $\mathcal{P} \subset \obj \cu$ we will write $\lan \mathcal{P} \ra$ for the smallest full retraction-closed triangulated subcategory $\du$ of $\cu$ such that $\mathcal{P} \subset \obj\du$. We will  call  $\du$  the triangulated category {\it densely generated} by $\mathcal{P}$. 

\item For any $A,B,C \in \obj \cu$ we will say that C is an extension of B by A if there exists a distinguished triangle $A \to C \to B \to A[1]$. A class
$\mathcal{P}\subset \obj \cu$ is said to be extension-closed if it is closed with respect to extensions and contains $0$.

%\item The smallest extension-closed retraction-closed class $\mathcal{P'}\subset \obj\cu$ containing $\mathcal{P}$ will be called the envelope of $\mathcal{P}$.

\item For $X,Y\in \obj \cu$ we will write $X\perp Y$ if $\cu(X,Y)=\ns$. For
$D,E\subset \obj \cu$ we write $D\perp E$ if $X\perp Y$ for all $X\in D,\
Y\in E$.
Given $D\subset\obj \cu$ we  will write $D^\perp$ for the class
$$\{Y\in \obj \cu:\ X\perp Y\ \forall X\in D\}.$$
%Sometimes we will denote by $D^\perp$ the corresponding full subcategory of $\cu$. 
Dually, ${}^\perp{}D$ is the class
$\{Y\in \obj \cu:\ Y\perp X\ \forall X\in D\}$.

\item Given $f\in\cu (X,Y)$ %, where $X,Y\in\obj\cu$, 
 we will call the third vertex
of (any) distinguished triangle $X\stackrel{f}{\to}Y\to Z$ a {\it cone} of
$f$.\footnote{Recall %{del}
that different choices of cones are connected by non-unique isomorphisms.}\

\item All coproducts in this paper will be small.

%\item Let $\hu$ be a (not necessarily additive) subcategory of an additive category $\bu$. We will call the full additive subcategory of $\bu$ whose objects are all
%retracts of arbitrary (small) coproducts of objects of $\hu$ in $\bu$ the coproductive hull of $\hu$ in $\bu$; we will write $\hu^{\widehat\oplus}$ for it.

\item Assume that $\cu$ is {\it smashing}, that is, closed with respect to %(small) 
 coproducts. For $\du\subset \cu$ ($\du$ is a triangulated
category that may be equal to $\cu$) one says that $\mathcal{P}\subset \obj \cu$ generates $\du$ as a {\it localizing subcategory} of $\cu$ if $\du$ is the smallest full strict triangulated
subcategory of $\cu$ that contains $\mathcal{P}$ and is closed with respect to $\cu$-coproducts.

\item $M\in \obj\cu$ is said to be {\it compact} if the functor $\cu(M, -): \cu \to \ab$ respects coproducts.

\item $\cu$ is said to be {\it compactly generated} if it is generated by a set of compact objects as its own localizing subcategory.
\end{itemize}

\subsection{Basics on $t$-structures }\label{wat}

Let us recall some notations and properties on $t-$structures. In contrast to the %classical paper 
 original definitions in \cite{bbd}, our convention for $t-$structures will 
be homological.\footnote{Here we follow \cite{Mor}. %Clearly, the 
 The homological and cohomological convention are %connected
 related  in the usual way: $\cu^{t\leq n}=\cu_{t\geq -n}$ and $\cu^{t\geq n}=\cu_{t\leq -n}$.} %\ (cf. \cite{Mor}, \cite{bac},\cite{bondegl}, \cite{degormod}, and \cite{hoy}).

\begin{defi}\label{dtst}

 A pair of strict subcategories $\cu_{t\ge 0},\cu_{t< 0}\subset\obj \cu$ will be said to define a
$t$-structure $t$ on a triangulated category $\cu$ if they  satisfy the following conditions.

(i) $\cu_{t\ge 0}[1] \subset \cu_{t\ge 0}$ and $\cu_{t<0}[-1] \subset \cu_{t<0}$.

(ii) $\cu_{t\ge 0}\perp \cu_{t< 0}$.

(iii) For any $M\in\obj \cu$ there
exists a distinguished triangle
$$M_{t\ge0}\to M\to M_{t<0} {\to} M_{t\ge0}[1]$$
such that $M_{t\ge0}\in \cu_{t\ge 0},\   M_{t<0}\in \cu_{t< 0}$.
\end{defi}

We also need the following auxiliary definitions.

\begin{defi}\label{dotstr}
\begin{enumerate}

\item $\cu_{t\ge n}:= \cu_{t\ge 0}[n]$ (resp. $\cu_{t \le n+1}:= \cu_{t< n}:= \cu_{t<0}[n]$) for any integer $n \in \mathbb{Z}$.

\item The heart of $t$ is the category $\underline{Ht}=\cu_{t\ge 0}\cap \cu_{t\le 0}\subset \cu$; recall that it is %necessarily 
 an abelian category.

\item We will say that $t$ is left (resp. right) {\it non-degenerate} if $\cap_{i \in \mathbb{Z}}\cu_{t\leq i}=\{0\}$ (resp.  $\cap_{i \in \mathbb{Z}}\cu_{t\ge i}=\{0\}$). We  say that $t$ is non-degenerate if it is both left and right non-degenerate.

%\item Assume that $\cu$ smashing and cosmashing. We will say that $t$ is left (resp. right) {\it complete} if the canonical maps $E \mapsto holim_{n} \tau_{\leq n}(E)$ (resp. $hocolim_{n}\tau_{\geq n}(E) \mapsto E$) are isomorphisms for all $E \in \cu$ 
%(see \S1.6 of \cite{{neebook}} for the definition of $ho(co)lim$).

\end{enumerate}
\end{defi}

\begin{rema}\label{ronts}

1. The triangle in %\ref{dtst}
axiom (iii) is essentially functorial in $M$. Thus, we get a well-defined functor $\tau_{\geq 0}: \cu \to \cu_{t\geq 0}$ (resp. $\tau_{<0}: \cu \to \cu_{t<0}$) which is right (resp. left) adjoint to the inclusion functor $\cu_{t\geq 0}\hookrightarrow \cu$ (resp. $\cu_{t< 0}\hookrightarrow \cu$. Also we put $\tau_{\geq n}(M):=\tau_{\geq 0}(M[-n])[n]$ (resp. $\tau_{\leq n}(M):=\tau_{<0}(M[n+1])[-n-1]$).

2. The  functor $H_{0}^{t}:=\tau_{\geq 0}\circ \tau_{\leq 0}$ sends $\cu$ into  $\underline{Ht}$; it is homological (i.e., converts distinguished triangles into long exact sequences). %We also
 Moreover, we will write $H_{n}^{t}$ for $H_{0}^{t}\circ [-n]$.

3. %Clearly,
 One can easily check that $t$ is non-degenerate if and only if the family of functors $(H_{n}^{t})_{n\in \mathbb{Z}}$ is conservative.
 
 4. Consider two categories $C$ and $D$ endowed with $t-$structures. One says that a functor $F: C \to D$ is %{\it left $t-$exact} (resp. right $t-$exact)
 {\it left} (resp. {\it right) $t-$exact}  if it respects the $t$-negativity (resp. $t$-positivity) of objects. We will say that $F$ is {\it $t-$exact}
if it is both left and right $t-$exact.
\end{rema}

\begin{pr}\label{cotsfc}
Let $\mathcal{P}\subset \obj\cu$ be a set of compact objects. Then there
exists a unique $t$-structure $t$ on $\cu$ such that $\cu_{t\geq 0}$ is the smallest subclass of
$\obj\cu$ that contains $\mathcal{P}$ and is stable with respect to extensions, the suspension $[1]$, and %(small) 
 coproducts.
\end{pr}

\begin{proof}
This is precisely Theorem A.1 of \cite{ajs}.
\end{proof}

\begin{defi}\label{tgdasj}
We will call the $t-$structure as in the previous proposition the $t-$structure {\it generated by} $\mathcal{P}$.
\end{defi}

\begin{rema}
1. For a $t-$structure  given by Proposition \ref{cotsfc} the functors $\tau_{\geq 0}, \tau_{\leq 0}$, and $H_{0}^{t}$ %commute with 
 respect coproducts (see Proposition A.2 of ibid.).\\
2. The $t-$structure as in the proposition is left non-degenerate if and only if $\mathcal{P}$ %is a generating family for
 generates  $\cu$ as its own localizing subcategory (obvious; see Lemma 1.2.9 of \cite{bondegl}).
\end{rema}

\subsection{On various motivic categories}\label{vgabc}

Now we recall some basics on  motivic categories. Below $k$ will always be a perfect field of characteristic $p$, and $\mathfrak{D}(k)$ is one of the motivic categories listed below;  
$\mathcal{M}=\mathcal{M}_{\mathfrak{D}}:\sv \to \mathfrak{D}(k)$ will denote the corresponding "$\mathfrak{D}-$motive" functor from the category $\sv$ of smooth $k$-varieties.
We will always assume that $\mathfrak{D}$ is triangulated monoidal with the tensor unit given by $\ood=\mathcal{M}_{\mathfrak{D}}(\spe(k))$. Moreover, we assume that $\mathcal{M}_{\mathfrak{D}}$ sends products of varieties into tensor products.

%All the categories discussed above are well known to be smashing and monoidal, with the tensor unit given by $\mathcal{M}(\spe(k))$. 
 
$\mathcal{M}_{\mathfrak{D}}$  will satisfy the {\it homotopy invariance} property, that is, $\mathcal{M}_{\mathfrak{D}}( \afo)\cong \ood$. %\mathcal{M}_{\mathfrak{D}}$.

\begin{rema}\label{rtate}
1. Now
%We will use the following notations for ''relevant twists'
 let us introduce some notation related to "Tate-type twists". 
 
 Firstly, %for any $C \in \mathfrak{D}(k)$ and $n\ge 0$  we set  $C\lan n \ra:=C\otimes T^{\wedge n}$, where 
  we set $T= \co(\mathcal{M}_{\mathfrak{D}}(\mathbb{G}_m) \to \mathcal{M}_{\mathfrak{D}}(\mathbb{A}^{1}))$. Below we will always assume that $T$ is $\otimes$-invertible in $\md$ (yet cf. Remark \ref{cgmus}(4) below).
For $C\in \obj {\mathfrak{D}}$ and $n\in \z$ we set $C\lan n \ra=C\otimes T^{\otimes n}$ and  $C\{n\}=C\lan n \ra[-n]$. 

2. These twists are clearly ''coherent  with respect to functors that commute with $\mathcal{M}$'' in various motivic categories below.

Moreover, homotopy invariance implies $\mathcal{M}_{\mathfrak{D}}(\mathbb{G}_m)\cong  \ood\bigoplus T[-1]$. Furthermore, it is well-known (and follows from the so-called Mayer-Viertoris property)  for the categories we consider that  $\mathcal{M}_{\mathfrak{D}}(\mathbb{P}^{1}) \cong \ood \oplus T$.

 These splittings also exist in the motivic categories of the type $\mathfrak{D}^{eff}$ that we discuss in Remark \ref{cgmus}(2-3) below.
\end{rema}

\begin{itemize}

\item\label{proshp} We will write $SH(k)$ for the $\mathbb{P}^1-$stable motivic homotopy category, and $\mathcal{M}_{SH}: \sv \to SH(k)$ for the corresponding infinite suspension spectrum functor (see \cite{Mor}, \S5.1).
%We write $SH(k)^{eff}$ for the localising subcategory of $SH(k)$  generated by objects of the form
%$\mathcal{M}(X)$ for $X \in \sv$ %(for example, see
% (cf. Definition 2.2.1 of \cite{bondegl}).

\item\label{prodm} We will write $DM(k)$ for the category of Voevodsky's motives, and $\mathcal{M}_{DM}: \sv \to DM(k)$ for the corresponding ("usual") $DM$-motive functor (see \S5.1 of \cite{degmod} for the detail). 
%We write $DM(k)^{eff}$ for the localising subcategory of  $DM(k)$  %its localising subcategory 
% generated by objects of the
%form $\mathcal{M}(X)$ for $X \in \sv$ (see \S4 of loc.cit. or \S1.4 of \cite{bokum}). Note that there is adjunction $\Sigma^{\infty}_{DM}: DM(k)^{eff} \rightleftarrows DM(k): \Omega^{\infty}_{DM}$ (see \S5.1 of \cite{degmod}).

\item\label{proda} We will write $D_{\mathbb{A}^{1}}(k)$ for the $\mathbb{A}^{1}-$derived category as defined in \S5.3.20 of \cite{cidtm}, and $\mathcal{M}_{D_{\mathbb{A}^{1}}}: \sv \to D_{\mathbb{A}^{1}}(k)$ for the corresponding functor (see also \S6.2 of \cite{Mor}).
%We write $D_{\mathbb{A}^{1}}^{eff}(k)$ for its localising subcategory generated by objects of the form
%$\mathcal{M}(X)$ for $X \in \sv$ (see \S\S5.3.19--20 of \cite{cidtm}). Once again, there is  an adjunction %of the form 
% $\Sigma^{\infty}_{D}: D_{\mathbb{A}^{1}}^{eff}(k) \rightleftarrows D_{\mathbb{A}^{1}}(k): \Omega^{\infty}_{D}$ (see \S5.3.23.2 of ibid.).

\item\label{promod} We will write $\mgl-\modd(k)$ for the category of $\mgl$-modules in $\sht$; see Propositions 7.2.14, 7.2.18 of \cite{cidtm}, Example 1.3.1(3) of \cite{bondegl}, or \S2.2 of \cite{degormod}. %Note that there is an adjunction
%$\mathcal{L}: SH(k) \rightleftarrows \mgl-\modd(k): \mathcal{O}$. Moreover, we have the following equality
%$\mathcal{L}(\mathcal{M}_{SH}(X)\{i\}[n])= \mgl(X)\{i\}[n]$, and the functor $\mathcal{O}$ is conservative and weakly monoidal.

%\item\label{proshs} 
\end{itemize}

\begin{rema}\label{cgmus}
1.  All the categories discussed above are well known to be smashing % and monoidal, with the tensor unit given by $\mathcal{M}(\spe(k))$. Moreover, all the categories $\mathfrak{D}=\mathfrak{D}(k)$ are
 and  compactly generated by the objects 
 $\mathcal{M}_{\mathfrak{D}}(X)\{i\}$  for $X\in \sv, i \in \mathbb{Z}$.
 
 Furthermore, the functors $\mathcal{M}_{\mathfrak{D}}$ factor through the corresponding subcategories of compact objects.

 2. For each of these  $\mathfrak{D}$ we will write $\mathfrak{D}^{eff}(k)$ for the localising subcategory of  $\mathfrak{D}(k)$ generated by objects of the
form $\mathcal{M}(X)$ for $X \in \sv$. Thus $\mathfrak{D}^{eff}(k)$ is also compactly generated by these objects.

3. We will write $SH^{S^{1}}(k)$ for the ${S}^1-$stable motivic homotopy category, and $\mathcal{M}_{SH^{S^{1}}}: \sv \to SH^{S^{1}}(k)$ for the corresponding infinite suspension spectrum functor (see \cite{Mor}, \S4.1).
The category $SH^{S^{1}}(k)$ is compactly generated by the objects $\mathcal{M}(X)$ for $X\in \sv$. Note that there exists an adjunction $\sigma: SH^{S^{1}}(k) \rightleftarrows SH(k): \omega$ (see Remark 5.1.11 of \cite{Mor}). %Also note that

% However, %note that this
  The  category $SH^{S^{1}}(k)$  is not equivalent to  the effective category $\mathfrak{D}^{eff}$ for any "motivic" $\mathfrak{D}$.  %(see Remark \ref{cgmus} below).
Nonetheless, by  abuse of notation we %certainly consider
 will take $SH^{S^{1}}(k)$ as %an 
  one of the possibilities for the category $\mathfrak{D}^{eff}(k)$  %effective category
   in all the %statements below 
    formulations in this paper except the ones in \S\ref{etfiloth}.   %and the essential image of the functor $\sigma$ is equivalent to $SH(k)^{eff}$ (e.g. see Example 2.3.13(2) of \cite{bondegl}).
   
  4. In   categories of the type   $\mathfrak{D}^{eff}(k)$ %(including $SH^{S^{1}}(k)$) the corresponding object $T$
   the twists of the types $-\lan n \ra$ and $-\{n\}$ are defined for $n\ge 0$ only. Note also that these functors are not fully faithful on  $SH^{S^{1}}(k)$ (in contrast to other $\mathfrak{D}^{eff}(k)$).

\end{rema}

\subsection{Homotopy $t$-structures: recollection}\label{thomaxc}

Now we introduce some definitions, which will be very important further in this paper.
We will write $\mathfrak{D}$ for  some of the motivic categories from section \ref{vgabc} and $\mathfrak{D}^{eff}$ for its effective version. 
%$\mathcal{M}_{\mathfrak{D}^{eff}}:\sv \to \mathfrak{D}^{eff}$ will denote the corresponding "$\mathfrak{D}-$motive" functor. 

\begin{defi}\label{thmdf}

1. Denote by $t_{hom}^{\mathfrak{D}}$ the $t-$structure on $\mathfrak{D}$ generated by $\mathcal{M}(X)\{i\}$ for $X \in \sv, i \in \mathbb{Z}$ (see Definition \ref{tgdasj}).

2. Denote by $t_{hom}^{\mathfrak{D}^{eff}}$ the $t-$structure on $\mathfrak{D}^{eff}$ generated by $\mathcal{M}(X)$ for $X \in \sv$.

3. For $i\in \z$,  $C \in \mathfrak{D}$ and $j\in \z$  (resp. $C \in \mathfrak{D}^{eff}$ and $j\ge 0$) we define the %following
  functor $C^{i}_{j}(-): \sv^{op} \to \ab$ as $X \mapsto \mathfrak{D}(\mathcal{M}(X)\{j\}, C[i])$ 
(resp. $X \mapsto \mathfrak{D}^{eff}(\mathcal{M}(X)\{j\}, C[i])$).

More generally, for any cohomological functor from $\mathfrak{D}(k)$ to $\underline{Ab}$ and $X \in \sv$, $i, j \in \mathbb{Z}$ we define $H_{j}^{i}(\mathcal{M}(X))$ as $H(\mathcal{M}(X)\{j\}[-i])$.

4. %Below 
 For $i,j$ as above we will write $\mathfrak{D}^{eff}(\mathcal{M}(K)\{j\}[i], -)$ (resp. $\mathfrak{D}(\mathcal{M}(K)\{j\}[i], -)$) for the following functor from $\mathfrak{D}^{eff}$: $ C \mapsto \varinjlim \limits_{X, k(X)=K}C^{-i}_{j}(X)$ (resp. for the functor
on $\mathfrak{D}$: 
$C \mapsto\varinjlim \limits_{X, k(X)=K}C^{-i}_{j}(X)$).

5. Denote by $-_{-1}$ the right adjoint to $-\{1\}$; %(\mathcal{M}(\mathbb{G}_m)/\mathcal{M}(\spe(k))$ (it is $T$ in the notation of the beginning of \S\ref{vgabc})
this is the so-called Voevodsky contraction.\footnote{Essentially following Definition 4.3.10 of \cite{Mor}; %, instead of the original definition from
note that the latter is equivalent to the original definition from \cite{vo1}. %Its existence follows from the fact that
 This adjoint exists since %all our categories are 
 $\mathfrak{D}^{eff}$ is compactly generated (see Remark \ref{cgmus}(2,3)), and  $-\{1\}$ is exact and respects coproducts.} 
For $i>0$ the $i$-th iteration of $-_{-1}$ will be denoted by  $-_{-i}$. 
\end{defi}

We introduce some ''axioms'' characterizing homotopy $t-$structures on  motivic categories that we consider in this paper.\\

\textbf{(A1)}  Let $H$ be a cohomological functor from $\mathfrak{D}(k)$ to $\underline{Ab}$ and $X\in \sv$. Then %we have the following 
 there exists a convergent  (coniveau) spectral sequence as follows: $$E_{1}^{p,q}=\underset{x\in X^{(p)}}{\coprod}H^{q}_{p}(x) \Rightarrow H^{p+q}_{0}(\mathcal{M}(X)),$$ 
where $X^{p}$ is the set of points of $X$ of codimension $p\ge 0$, and for a presentation of $x \in X^{p}$ as $\varprojlim_{i}X_{i}$ for $X_{i} \in \sv$ we define $H^{q}_{p}(x)$ as 
$\varinjlim_{i} H(\mathcal{M}(X_{i})\{p\}[-q])$.\\
%This is a convergent spectral sequence.\\
 
\textbf{(A2)} $\mathfrak{D}(k)_{t_{hom} \geq 0}= \{C \in \mathfrak{D}(k) | \mathfrak{D}(\mathcal{M}(K)\{j\}, C[i])=\ns$ for all function fields $K/k$, $j \in \mathbb{Z}$, $i>0\}$.\\

%Clearly, %absolutely analogous
 %very similar axioms are fulfilled in the effective case.
In the effective setting one should take $j \ge 0$ in Axiom (A2) instead.

\begin{coro}\label{conspecseq}
\begin{enumerate}
\item\label{axcor0} The endofunctor $(-)_{-1}$ is $t_{hom}-$exact.

\item\label{axcor1}The functor $\Phi: \underline{Ht}^{\mathfrak{D}} \to Psh(Pts, Ab^\mathbb{Z}), \Phi(F)= F_{*}^{0}(-)$ (see the related definitions above) is conservative, exact, and commutes with  coproducts. %and twists by $\{i\}$ in the appropriate sense.
 Here we write $Pts$ for the set of function fields (that is, finitely generated extensions) over the base field $k$.
 
%Similarly for 
 Moreover, the functor $\Phi_0:F\mapsto F_{0}^{0}(-)$ is conservative on $\underline{Ht}^{\mathfrak{D}^{eff}}$.

\item\label{axcor2} $t_{hom}^{\mathfrak{D}}$ and $t_{hom}^{\mathfrak{D}^{eff}}$ are non-degenerate (see Definition \ref{dotstr}(3)).
\end{enumerate}
\end{coro}
\begin{proof}
1. The right $t_{hom}-$exactness follows from \textbf{(A2)}. Next we note that $-\{1\}$ is right $t_{hom}-$exact, thus $(-)_{-1}$ is also left $t_{hom}-$exact (as a right adjoint).

2. The description of  $\mathfrak{D}(k)_{t_{hom} \geq 0}$ provided by \textbf{(A2)} immediately yields that $\Phi(F)\neq 0$ if $F\in  \mathfrak{D}(k)_{t_{hom} \geq 0}\setminus \mathfrak{D}(k)_{t_{hom} \geq 1}$.
Next, $\Phi$ commutes with coproducts, since it is defined in terms of functors corepresented by compact objects.

 To prove  exactness, we apply the functor $\Phi$ to the exact sequence $0 \to F' \to F \to F'' \to 0$ in $\underline{Ht}^{\mathfrak{D}}$.  Our definitions immediately imply  $\Phi(F''[-1])=0$, whereas %It remains to note that 
  $\Phi(F'[1])=0$ by \textbf{(A2)}. 
  
We should also prove that if $F\in\obj \underline{Ht}^{\mathfrak{D}^{eff}}$ and $\Phi_0(F)=0$ then $F_{i}^{0}(-)=\ns$ for $i>0$. %Next 
 Firstly, %we have the following group
   $F_{*}^{*}(\mathcal{M}((\mathbb{G}_{m,K})^{\times i}))$ equals $\varinjlim_{j} F_{*}^{*}(\mathcal{M}((\mathbb{G}_{m})^{\times i}(X_{j})))$ for a presentation
$Spec(K)=\varprojlim_{j}X_{j}$. % and %we have to
  We should  prove that $F^{p+q}_{0}(\mathcal{M}((\mathbb{G}_{m,K})^{\times i}))=\ns$ for $p+q=0$. We take the spectral sequence as in \textbf{(A1)} for the variety $(\mathbb{G}_{m})^{\times i}(X_{j})$ for each $j$.
  By our assumption and \textbf{(A2)}, %all these groups in the spectral sequence are zero for 
   the corresponding $E_{1}^{p,q}$ %=\ns$ 
    vanish if  $q>0$ and $p=q=0$. Further, for $q<0$ (i.e., in
the case $p=-q$) these groups %are also 
 vanish by the property (ii) for $t-$structures and the right $t-$exactness of $-\{p\}$.
Thus  $ F_0^0(\mathcal{M}((\mathbb{G}_{m})^{\times i}(X_{j})))=0$ for any $j$, 
  and passing to the direct limit we obtain our assertion.% spectral sequence as in \textbf{(A1)}
 
3. Can be easily obtained by combining
  a spectral sequence argument (Axiom \textbf{(A1)}) with the corresponding compact generation properties; see Remarks \ref{ronts}(3) and \ref{cgmus}(1--2).
\end{proof}

It will be shown below that axioms (\textbf{A1}--\textbf{A2}) hold for %various 
 all aforementioned motivic categories (and for some other ones).

\subsection{The case of $SH(k)$ and $SH^{eff}(k)$}\label{parcasosh}

In this section, we check our axioms and discuss the related definitions for the categories $SH^{eff}(k) \subset SH(k)$ that appear to %yield 
 give the most interesting examples. % one.

\begin{defi}
We will write $\underline{\pi}_{i}(E)_{j}$ for the Nisnevich  sheaf on $\sv$ that is associated to the presheaf %(in the Nisnevich topology)
  $E^{i}_{j}$.%{\pi}_{i}(E)_{j}(X) = SH(k)(\mathcal(M)(X_{+})[i], E\{j\})$.
\end{defi}

\begin{lem}\label{mlfjooa}
1. There exist convergent spectral sequences as in  \textbf{(A1)} for these   %settings. %
cases.

2. %The 
Axiom \textbf{(A2)} is fulfilled.
\end{lem}
\begin{proof}
1. See Proposition 4.3.1(I.3) and Remark 4.3.2(2) of \cite{bgn} (note that the construction there was inspired by the corresponding results from \cite{chk}).

2. See Proposition 5.1.1(5) of \cite{bgn}.
\end{proof}

\begin{rema}
 Thus, our axioms \textbf{(A1)} and \textbf{(A2)} above are fulfilled for $SH(k)$ and $SH^{eff}(k)$.  Therefore, Corollary \ref{conspecseq} can be applied in this case as well. 
\end{rema}

\begin{pr}\label{thshc}

\begin{enumerate}
\item $SH(k)_{\geq 0} = \{E \in SH(k) | \underline{\pi}_{i}(E)_{j}=0 \ for\  i>0, j \in \mathbb{Z}\}$.
\item $SH(k)_{\leq 0} = \{E \in SH(k) | \underline{\pi}_{i}(E)_{j}=0\ for\ i<0, j \in \mathbb{Z}\}$.
\item There is an adjunction $ i^{SH}: SH^{eff}(k) \rightleftarrows SH(k): \omega^{SH}$. Moreover,  $i^{SH}$ is right $t_{hom}-$exact, and $\omega^{SH}$ is  $t_{hom}-$exact.
\item The functor $E \mapsto \underline{\pi}_{0}(E)_{*}$ induces an equivalence of categories $\hrt^{SH}$ and $HI_{*}(k)$, where $HI_{*}(k)$ is the category of homotopy modules (see Definition 5.2.4 of \cite{Mor}). Next, the functor $E \mapsto \underline{\pi}_{0}(E)_{0}$
induces an equivalence $\hrt^{SH^{eff}} \cong HI^{fr}(k)$, where $HI^{fr}(k)$ is the category of homotopy invariant stable Nisnevich sheaves with framed transfers (see \S1 of \cite{gap} for this definitions).
\item For $E \in \underline{Ht}$, $X \in \sv$, and $n \in \mathbb{Z}$ we have $H_{Nis}^{n}(X, \underline{\pi}_{0}(E)_{0}) \cong E^{n}_{0}(X)$.
\end{enumerate}
\end{pr}

\begin{proof}
1,2. This has been proved in Theorem 2.3 of \cite{hoy}.\\
%3. See Corollary \ref{conspecseq}(\ref{axcor2}), cf. also \S5.2 of \cite{Mor} (cf. also Corollary 3.3.7(1) of \cite{bondegl} for relative version).\\
3. See Corollary 3.3.7(2) of \cite{bondegl} or Proposition 2.2.5(1) of \cite{binfeff}.\\
4. See Theorem 5.2.6 of \cite{Mor} and Theorem 5.14 of \cite{by}.\\
%6. See Corollary \ref{conspecseq}(\ref{axcor1}) and Lemma \ref{mlfjooa} above.\\%See Proposition 4.2.2 of \cite{bondegl} (cf. also Theorem 5.2.3(6) of \cite{bgern}).
%5. See Corollary \ref{conspecseq}(\ref{axcor0}).\\        %Easy via the coniveau spectral sequence (see  Lemma \ref{amofgw} for a similar statement), cf. also Theorem 5.2.3(6) of \cite{bgn}.\\
5. This fact is well-known; see Proposition 5.1.1(8) of \cite{bgn}.\\
\end{proof}

\section{"Smooth" aisles and birationality filtrations in motivic categories}\label{coresul}

This section contains the central technical results of the paper.\\
In \S\ref{kbmrotp} we define certain aisles in terms of motives of smooth varieties and  describe them %through
 in terms of the %Zariski 
 stalks corresponding to function fields. Next we prove that $t_{hom}-$homology respect these classes.\\
In \S\ref{wocvs} we introduce the $n-$birational categories $\mathfrak{D}^{n-bir}$. %Next %using our ''smooth'' 
 We use our aisles to  study %some conditions on 
  the birationality filtration on $\hrt^{eff}_{hom}$, and % in these terms.   and 
   prove that the homotopy $t-$structure restricts to $\mathfrak{D}^{n-bir}\subset \mathfrak{D}^{eff}$. 

\subsection{Smoothly generated aisles: definition and main properties}\label{kbmrotp} 

\begin{defi}\label{naortcon}
If $s=(s_{j})$ is  a non-decreasing sequence in $\mathbb{Z}\cup \{\pm \infty\}$, $j \in \mathbb{Z}$,    then we define an %following aisles (see Remark \ref{aiaiais} below):
aisle (see Remark \ref{aiaiais} below)  $\mathcal{A}_{s}$ as $(\{\mathcal{M}_{R}(\sv)\lan j \ra [t_{j}]\})^{\perp_{\mathfrak{D}}}$, where we take all $t_{j}<s_{j}$, $t_j\in \z$.

We %also consider a similar definition for 
 will also use the %obvious 
  following modification of this definition in the effective setting: %taking 
   we take $j \geq 0$ and take the orthogonal in $\mathfrak{D}^{eff}$.  %and  with $j \geq 0$.
\end{defi}

\begin{lem}\label{amofgw}
1. We have the following equality: $\mathcal{A}_{s}= \{C \in \mathfrak{D} | \mathfrak{D}(\mathcal{M}(K)\lan j \ra[t_{i}], C)=0\}$ for all function fields $K/k$, $j \in \mathbb{Z}$, and $t_{j}<s_{j}$ (see Definition \ref{thmdf}(4) for the corresponding notation).

2. Similar assertion holds in the effective case if one takes $j \geq 0$. %we have the following: $\mathfrak{D}_{\wchosw \geq 0}= \{C \in \mathfrak{D} | \mathfrak{D}(\mathcal{M}(K)\lan n \ra[i], C)=0$ for all function fields $K/k$, $n \in \mathbb{Z}, i<0\}$.

\end{lem}

\begin{proof}
1. We  should prove that $\mathcal{A}_{s}=(\{\cup_{j\in \mathbb{Z}}\mathcal{M}_{R}(K)\lan j \ra[t_{i}]\})^{\perp_{\mathfrak{D}}}$. 
Obviously, the first of these classes is contained in the second one; cf. % (since the spectra of function fields are the henselizations of  smooth varieties at generic points)
 Definition \ref{thmdf}(4). Conversely, if  $N$ belongs to the second class
%To check the inverse inclusion we consider the
 then the convergent coniveau spectral sequence  $$E_{1}^{p,q}=\underset{x\in X^{(p)}}{\bigoplus}H^{q}_{j+p-t_{j}}(x) \Rightarrow H^{p+q-t_{j}}_{j}(\mathcal{M}(X))$$ from \textbf{(A1)} yields that it belongs to $\mathcal{A}_{s}$;
 here $X^{p}$ is the set of points 
of $X$ of codimension $p$, and $H^{q}_{p}(-)$ are the following cohomology theories (on $\sv$): $H^{q}_{p}(X):= \mathfrak{D}(\mathcal{M}_{R}(X)\lan p \ra [-p-q], N)$ (see Definition \ref{thmdf}(3)).
 %This is a converging spectral sequence; thus we obtain the implication in question.

2. In the effective case the proof is completely similar.
\end{proof}

%Now we %prove some results on the comparison of
% relate  $\wchosw$ to $t_{hom}$ (as well as their effective versions).

\begin{pr}\label{wcatheq}
The following conditions for an object $C \in \mathfrak{D}$  are equivalent:\\
1. $C \in \mathcal{A}_{s}$.\\
2. $H_{r}^{t_{hom}}(C)[r] \in \mathcal{A}_{s}$ for all $r \in \mathbb{Z}$. 

%The analogous
 %A similar  statement %(with $r\ge 0$) 
The obvious effective version of this statement is valid as well. %also fulfilled in the effective case.
\end{pr}

\begin{proof}
 By the previous lemma %we have
  it suffices  to prove that %Zariski  the corstalks are the same, i.e. 
  $\mathfrak{D}(\mathcal{M}(K)\{ j \}[t_{i}], C) \cong \mathfrak{D}(\mathcal{M}(K)\{ j \}, H_{t_{i}}^{t_{hom}}(C))$ for all function fields $K/k$ and $ j \in \mathbb{Z}$. %but this equality 
 This isomorphism is well known. By the definition of  %holds for $j\geq 0$ by definition of 
  $t_{hom}^{\mathfrak{D}}$ (see Proposition \ref{thshc}(1,2)) and the adjunction from Remark \ref{ronts} it suffices to verify it for $C\in \mathfrak{D}_{t_{hom} \leq t_{i}}$, and in the latter case it follows from  %  and for $j< 0$ by 
   the well-known Corollary 2.4 of \cite{hoy}. % (which is well-known
%isomorphic to $\mathfrak{D}^{eff}(\mathcal{M}(K)\lan n \ra[j], \tau_{\leq 0}^{t_{hom}}(C))$ and 
%\underline{2 $\Rightarrow$ 1}. This implication follows from the well-known spectral sequence: $$E_{2}^{p,q} = \mathfrak{D}(\mathcal{M}(K) \lan j \ra[t_{j}], H_{q}^{t_{hom}}(C)[p]) \Rightarrow \mathfrak{D}(\mathcal{M}(K) \lan j \ra[t_{j}], C[p-q]).$$

In the effective case, the proof is literally the same.
\end{proof}

\begin{rema}\label{aiaiais}
1. It is noteworthy that our $\mathcal{A}_{s}$ are {\it aisles} (see \S1 of \cite{ajs} for the definition and discussions on this notion). Indeed, the pair 
$(\mathcal{A}_{s}, \mathcal{A}_{s}^{\perp})$ %[1])$??
 gives a $t-$structure for each $\mathcal{A}_{s}$; see \S\ref{wsmooo} below.

2. Note also that the assumption  that the sequence $s$ in  Definition \ref{naortcon} is non-decreasing is not restrictive. Indeed, %the corresponding  orthogonal for a general sequence $s$ is equal to the orthogonal for the  %corresponding 
 %"associated" non-decreasing one. Namely,
  if we take  $s'_j= \operatorname{sup}_{i \le j} s_i $ (for the corresponding values of $j$)  %as $\operatorname{sup}_{i \le j} s_i $, 
 then our definition easily implies the equality $\mathcal{A}_{s'} = \mathcal{A}_{s}$. %Indeed, for 
  To obtain the non-obvious inclusion here
 it suffices to recall that $\mathcal{M}(X) \lan n \ra$ is a retract of $\mathcal{M}(\mathbb{P}^{n}_{X})$ for any $X \in \sv$.
 
 For this reason, we only consider non-decreasing sequences; this also simplifies some of the formulations 
(see Propositions \ref{becares} and \ref{wawwes} below).
\end{rema}

\subsection{Weakly birational categories in terms of smooth aisles}\label{wocvs}
In this section we apply our results to shift-stable aisles.

\begin{itemize}
\item Let $n\ge 0$. Let us recall some %definitions and 
 properties of the $n-$birational motivic category $\mathfrak{D}^{n-bir}:=\mathfrak{D}^{eff}\{n+1\}^{\perp}$. Firstly,
we have an adjunction $p_{n}: \mathfrak{D}^{eff} \rightleftarrows \mathfrak{D}^{eff}/\mathfrak{D}^{eff}\{n+1\}: i^{(n)}$; here $p_{n}$ is the corresponding localization, the functor 
$i^{(n)}$ is fully faithful and induces an equivalence $\mathfrak{D}^{eff}/\mathfrak{D}^{eff}\{n+1\} \cong \mathfrak{D}^{n-bir}$, and  there exists a right adjoint $R_{nr, n}$ to it.\footnote{The 
 motivation for our terminology is that the localization functor $\mathcal{M}^{n-bir}_{R}:=p_{n}\circ \mathcal{M}_{R}$ sends 
  all open immersions $U \to V$ with  $codim_{V}(V\setminus U) \geq n+1$ into isomorphisms.} These statements 
 follow from well-known abstract nonsense; see Proposition 3.6 and Theorem 1.4(2) of \cite{pel} or Theorem A.2.6 and Lemma 4.5.4 of \cite{kabir}.

\item Also we recall very briefly the notion of so-called {\it slices}; see \S1 of \cite{pel} and \S4.2 of \cite{kabir} for details. We have an adjunction $i_{n} : \mathfrak{D}^{eff}\{n\} \rightleftarrows \mathfrak{D}^{eff} : r_{n}$, which defines a functor
$\nu^{\geq n} : \mathfrak{D}^{eff} \to \mathfrak{D}^{eff}$ as the composition $i_{n} \circ r_{n}$. 
Then we have the following {\it slice filtration} triangle for any $C \in \mathfrak{D}^{eff}$: 
\begin{equation}\label{sls} \nu^{\geq n}(C) \to C \to i^{(n-1)}p_{n-1}(C) \to \nu^{\geq n}(C)[1] \end{equation}
\end{itemize}

Next we study the properties of objects from the heart of $t_{hom}^{eff}$ with respect to the $\mathcal{A}_{s}-$filtration. Recall that objects of $\hrt^{eff}_{hom}$ for all categories considered in this article %, are objects 
are   %''sheaf nature''
  "sheaf-like" (see Propositions \ref{thshc}(4), \ref{thpragg}(2), \ref{thsss}(3), \ref{thdera}(3) for the corresponding descriptions). %Thus, below we will make no distinction between objects of $Ht^{\mathfrak{D}^{eff}}_{hom}$ and 
% objects of equivalent categories (i.e., $SHI(k), HI^{fr}(k)$ and $HI(k)$).

\begin{pr}\label{epowbo}
 Assume that $r\ge -1$ is fixed, and define a sequence $s^{r-bir}$ as follows: $s_{j}^{r-bir} = -\infty$ for $0\le j \le r$ and %the following $\mathcal{A}_{s_{j}}$ : 
  $s_{j}^{r-bir}= +\infty$ for $j \geq r+1$. Then for $S \in \obj \hrt^{eff}_{hom} \subset \mathfrak{D}^{eff}$ the following conditions are equivalent:\\
1. $S \in \mathcal{A}_{s}$.\\
2. The Nisnevich cohomology of $S$ vanishes in degrees $> r$.\\
3. $S(K\{m\})=\ns$ for all function fields $K/k$ and $m>r$, where $S(K\{m\})=\mathfrak{D}^{eff}(\mathcal{M}(K)\{m\}, S)$.\\
4. $S(K\{r+1\})=\ns$ for all function fields $K/k$.\\
5. $S_{-r-1}=0$ (see Definition \ref{thmdf}(5) for this notation). 

%Everywhere below, we denote such sequence by $s^{r-bir}$.
\end{pr}

\begin{proof}
%Consider the spectral sequence from \textbf{(A1)} : 
Axiom \textbf{(A2)} along with  Proposition \ref{amofgw}  immediately give  the equivalence of conditions (1) and (3). Recalling  
the spectral sequence  $$E_{1}^{p,q}=\underset{x\in X^{(p)}}{\coprod}S^{q}_{p}(x) \Rightarrow S^{p+q}_{0}(\mathcal{M}(X))$$  given by \textbf{(A1)} (here $X^{p}$ is the set of points of $X$ of codimension $p$). Now we note that 
$\mathfrak{D}^{eff}(\mathcal{M}(X), S[p+q]) \cong H^{p+q}_{Nis}(X, S)$ (see Proposition \ref{thshc}(5) and Remark \ref{oprnci}).
It easily follows that $(3) \Rightarrow (2)$. 
 
 Next, recall that  $S(K\{m\})=\varinjlim \limits_{X, k(X)=K}\md(\mmd(X)\otimes T^m[-m], S)$. Since $\md(\mmd(X)\otimes T^m, S)$ is a retract of $\md(\mmd((\p^1)^m(X)), S[m])$
  (see Remark \ref{rtate}(2)), we obtain  $(2) \Rightarrow (3)$. 
 
 Moreover, condition (3) obviously implies (4). Next, if (4) is fulfilled then $S_{-r-1}=0$ since we can apply Corollary \ref{conspecseq}(\ref{axcor1}) to $S_{-r-1}$. 
 Lastly, if (5) is valid then  $S_{-m}=0$ for all $m>r$; hence $S_{-m}(K)=\ns$ for all function fields $K/k$ by Axiom \textbf{(A2)}.
\end{proof}

\begin{rema}\label{eot}
1. It %can be
 is easily seen that $\mathcal{A}_{s^{r-bir}}=\mathfrak{D}^{eff}\{r+1\}^{\perp}$; see Remarks \ref{cgmus}(2,3).

The  objects of $\mathfrak{D}^{eff}\{r+1\}^{\perp}$ (that also satisfy the equivalent conditions of the proposition above) will be called {\it $r$-birational}. We will also say that they are {\it weakly birational}. 
Note also that our Proposition \ref{epowbo} generalizes Proposition 2.5.2 of \cite{kabir} (which follows from it if we put $r=0$).\\ 
 2. Clearly, %the ("degenerate") case of our statement corresponding to  %$r=-1$ and 
 conditions 1, 2, and 3 of the proposition are equivalent in the  ("degenerate") case $r=+\infty$  as well. %Of course, under the conditions of the proposition, the equivalence $\mathcal{A}_{s^{r-bir}}=\mathfrak{D}^{eff}\{r+1\}^{\perp}$ is valid.
 
%Furthermore, our definition gives the equality $\mathcal{A}_{s^{r-bir}}=\mathfrak{D}^{eff}\{r+1\}^{\perp}$.
%3. In the proof of $(4) \Rightarrow (3)$ above we have implicitly $S_{-r-1}=\ns$ if and only i
%This retraction argument (cf. also of \cite{bgn})  (3') $S(K\{r+1\})=\ns$ for all function fields $K/k$.
\end{rema}

\begin{pr}\label{thowbicgk}
%The
  $t^{eff}_{hom}$ % induces
 restricts to a (homotopy) $t$-structure $t_{hom}^{n-bir}$ on the category $\mathfrak{D}^{n-bir}$. Respectively, the functor $p_{n}$ is right $t-$exact, whereas $i^{(n)}$ is $t-$exact and $R_{nr, n}$ is left $t-$exact.
\end{pr}

\begin{proof}
It suffices to prove that the $t_{hom}-$truncations preserve the class
$\mathfrak{D}^{n-bir}:= \mathfrak{D}^{eff}\{n+1\}^{\perp}$, and the latter fact follows from Proposition \ref{wcatheq} immediately.
\end{proof}

\begin{rema}
Our Proposition \ref{thowbicgk}  widely generalizes Theorem 4.4.1 of \cite{kabir} (only case of $DM^{eff}(k)$ and $n=0$ was considered there). Note also that the %methods 
 arguments used in ibid. cannot be applied for $n>0$.
\end{rema}

\section{Smooth weight and $t$-structures and their applications}\label{weismot}

In \S\ref{jdogws} we recall the basics of the theory of weight structures.\\
In \S\ref{wsmooo} we give the definition of smooth weight structures $w_{Smooth}^{s}$ and $w_{Smooth}^{eff, s}$ corresponding to certain sequences $s_j$, and prove their main properties. We also express the conditions of {\it weak birationality} in terms of our $w_{Smooth}^{eff, s}$. 
Next we define a filtration on the $\hrt^{hom}$ using the adjacent $t-$structure $t_{\smo}^{eff}$, and prove some of its interesting properties.
 We also apply our results to the study of unramified cohomology.\\
In \S\ref{etfiloth} we prove some weak weight-exactness property of the functor $-\lan n \ra$ and so we obtain that  $\wchose$ %and $t_{\smo}^{eff}$
  induces %some structures
   a weight structure on the $n-$birational category $\mathfrak{D}^{n-bir}$ (actually, this is also true for  $w_{Smooth}^{eff, s}$ under certain assumptions on $s=(s_{j})$).

\subsection{Weight structures: a short recollection}\label{jdogws} %basics?

Let us recall the definitions and some properties of weight structures.

\begin{defi}\label{dwstr}

I. A pair of subclasses $\cu_{w\le 0},\cu_{w\ge 0}\subset\obj \cu$ %(of {\it $w$-negative} and {\it $w$-positive} objects, respectively)
will be said to define a weight
structure $w$ for a triangulated category  $\cu$ if 
they  satisfy the following conditions.\footnote{ In the current paper we use the so-called homological convention for weight structures; 
whereas in \cite{bws} the cohomological convention was used. In the latter convention 
the roles of $\cu_{w\le 0}$ and $\cu_{w\ge 0}$ are interchanged, i.e., one considers $\cu^{w\le 0}=\cu_{w\ge 0}$ and $\cu^{w\ge 0}=\cu_{w\le 0}$.}

(i) $\cu_{w\ge 0},\cu_{w\le 0}$ are %additive and 
retraction-closed in $\cu$
(i.e., contain all $\cu$-retracts of their objects).

(ii) {\bf Semi-invariance with respect to translations.}

$\cu_{w\le 0}\subset \cu_{w\le 0}[1]$, $\cu_{w\ge 0}[1]\subset
\cu_{w\ge 0}$.

(iii) {\bf Orthogonality.}

$\cu_{w\le 0}\perp \cu_{w\ge 0}[1]$.

(iv) {\bf Weight decompositions}.

 For any $M\in\obj \cu$ there
exists a distinguished triangle
%\begin{equation}\label{wd}
$$X\to M\to Y%\stackrel{f}
{\to} X[1]$$
%\end{equation} 
%\end{equation} 
such that $X\in \cu_{w\le 0},\  Y\in \cu_{w\ge 0}[1]$.\end{defi}

We will also need the following definitions.

\begin{defi}\label{dwso}

Let $i,j\in \z$; assume that a triangulated category $\cu$ is endowed with a weight structure $w$.

\begin{enumerate}
\item\label{idh}
The full category $\hw\subset \cu$ whose object class is %are
$\cu_{w=0}=\cu_{w\ge 0}\cap \cu_{w\le 0}$  is called the {\it heart} of the weight structure $w$.

\item\label{id=i}
 $\cu_{w\ge i}$ (resp. $\cu_{w\le i}$, resp.
$\cu_{w= i}$) will denote $\cu_{w\ge
0}[i]$ (resp. $\cu_{w\le 0}[i]$, resp. $\cu_{w= 0}[i]$).

\item\label{wgen} We will say that a weight structure $w$ is {\it generated by} a class $\mathcal{P}\subset \obj\cu$ if $\cu_{w\geq0} = (\cup_{i>0} \mathcal{P}[-i])^\perp$.

%\item\label{id[ij]} $\cu_{[i,j]}$  denotes $\cu_{w\ge i}\cap \cu_{w\le j}$; so, this class  equals $\ns$ if $i>j$. 

%$\cu^b\subset \cu$ will be the category whose object class is $\cup_{i,j\in \z}\cu_{[i,j]}$.

%\item\label{idbo} We will  say that $(\cu,w)$ is {\it  bounded}  if $\cu^b=\cu$ (i.e., if $\cup_{i\in \z} \cu_{w\le i}=\obj \cu=\cup_{i\in \z} \cu_{w\ge i}$).

\item \label{lrbo} We will call $\cup_{i\in \z} \cu_{w\ge i}$ %(resp. $\cup_{i\in \z} \cu_{w\le i}$) 
the class of $w-$bounded below %(resp., $w-$bounded above) 
objects of $\cu$; it will be denoted by $\cu_{+}$. % (they will be denoted as $\cu_{+}$ and $\cu_{-}$ respectively).

\item\label{smash}
We will say that %weight structure $w$ on $\cu$ 
 $w$ is {\it smashing} if both the category $\cu$ and the class $\cu_{w\ge 0}$ is closed with respect to (small)  $\cu$-coproducts (cf. Proposition \ref{pbw}(\ref{leftsm})).

\item\label{idwe}
 Let %$\cu$ and 
  $\cu'$ be a triangulated category endowed with
 a weight structures $w'$; let $F:\cu\to \cu'$ be an exact functor.

$F$ is said to be  
{\it  weight-exact} 
(with respect to $w,w'$) if it maps
$\cu_{w\le 0}$ into $\cu'_{w'\le 0}$ and
sends $\cu_{w\ge 0}$ into $\cu'_{w'\ge 0}$. 

\item\label{idrest}
Let $\du$ be a full triangulated subcategory of $\cu$.

We will say that $w$ {\it restricts} to $\du$ whenever the couple $(\cu_{w\le 0}\cap \obj \du,\ \cu_{w\ge 0}\cap \obj \du)$ is a weight structure on $\du$.

\end{enumerate}
\end{defi}

\begin{rema}\label{rstws}

 A weight decomposition (of any $M\in \obj\cu$) is (almost) never canonical. 

Still for any $m\in \z$  axiom (iv) gives the existence of distinguished triangle \begin{equation}\label{ewd} w_{\le m}M\to M\to w_{\ge m+1}M \end{equation}  with some $ w_{\ge m+1}M\in \cu_{w\ge m+1}$ and $ w_{\le m}M\in \cu_{w\le m}$; we will call it an {\it $m$-weight decomposition} of $M$.

 We will often use this notation below (even though $w_{\ge m+1}M$ and $ w_{\le m}M$ are not canonically determined by $M$);
we will call any possible choice either of $w_{\ge m+1}M$ or of $ w_{\le m}M$ (for any $m\in \z$) a {\it weight truncation} of $M$.
Moreover, when we will write arrows of the type $w_{\le m}M\to M$ or $M\to w_{\ge m+1}M$ we will always assume that they come from some $m$-weight decomposition of $M$.  
\end{rema}

\begin{pr} \label{pbw}
Let $\cu$ be a triangulated category, $n\ge 0$; we will assume 
that $w$ is a fixed 
weight structure on $\cu$.

\begin{enumerate}

\item \label{idual}
The axiomatics of weight structures is self-dual, i.e., for $\du=\cu^{op}$
(so $\obj\du=\obj\cu$) there exists the (opposite)  weight
structure $w'$ for which $\du_{w'\le 0}=\cu_{w\ge 0}$ and
$\du_{w'\ge 0}=\cu_{w\le 0}$.

\item\label{leftsm}
 $\cu_{w\le 0}$ is closed with respect to $\cu$-coproducts. % that exist in $\cu$.

 \item\label{iort}
 $\cu_{w\ge 0}=(\cu_{w\le -1})^{\perp}$ and $\cu_{w\le 0}={}^{\perp} \cu_{w\ge 1}$. Thus if $w$ is generated by a class $\mathcal{P}$ then $\mathcal{P}\subset \cu_{w\leq 0}$.

\item\label{icompl} Let $ m\le l\in \z$, $X,X'\in \obj \cu$; fix certain weight decompositions
        of $X[-m]$ and $X'[-l]$. Then  any morphism
$g:X\to X'$ can be
extended 
to a commutative diagram of the corresponding distinguished triangles (see Remark \ref{rstws}(2)):
 $$\begin{CD} w_{\le m} X@>{}>>
X@>{}>> w_{\ge m+1}X\\
@VV{}V@VV{g}V@ VV{}V \\
w_{\le l} X'@>{}>>
X'@>{}>> w_{\ge l+1}X' \end{CD}
$$

Moreover, if $m<l$ then this extension is unique (provided that the rows are fixed).
%\item\label{iwe} Assume that $w'$ is a weight structure on  a triangulated category $\cu'$. Then an exact functor $F:\cu\to \cu'$ is weight-exact if and only if $F(\cu_{w=0})\subset \cu'_{w'=0}$.
\end{enumerate}
\end{pr}

\begin{proof}
See Remark 1.1.2(1), Proposition 1.3.3(1,2,5) and Lemma 1.5.1(1,2) of \cite{bws}.
\end{proof}

\begin{pr} \label{pcgws}
Assume that $\cu$ is a compactly generated category.
\begin{enumerate}
\item\label{pc1} Let $\mathcal{P}$ be a set of compact objects of $\cu$. Then there exists (a unique) weight structure on $\cu$ that is generated by $\mathcal{P}$, and $\mathcal{P} \subset \cu_{w\leq0}$.

\item\label{pc2} Assume that an exact functor $F: C \to C'$ respects coproducts, $w$ is a weight structure on $\cu$ that is generated by some class $\mathcal{P}\subset \obj\cu$, and $w'$ is a weight structure on $\cu'$. 
Then $F$ is left weight-exact if and only if $F(\mathcal{P}) \subset \cu_{w\leq0}$.

%\item\label{pc3} Assume in addition that $F$ is surjective on objects. Then $F$ is weight-exact if and only if $w'$ is generated by $F(\mathcal{P})$.

\end{enumerate}
\end{pr}

\begin{proof}
See Theorem 5 of \cite{konk} (cf. also Proposition 1.2.3(II) of \cite{bokum}).
\end{proof}

\subsection{Smooth weight and $t-$structures, unramified cohomology, and the weakly birational filtration}\label{wsmooo}

\begin{defi}\label{tgwss}
For any category of the type $\mathfrak{D}^{eff}$ or $\mathfrak{D}$ 
we will write $w_{Smooth}^{eff, s}$ (resp. $w_{Smooth}^{s}$) for the weight structure defined by $\mathcal{A}_{s}$ on $\mathfrak{D}^{eff}$ (resp. on $\mathfrak{D}$); that is, we assume that the class of weight-non-negative objects %is equal to
 equals $\mathcal{A}_{s}$.

 We will write just $\wchose$ and $\wchosw$  in the case where all $s_j$ are zero. 
\end{defi}

\begin{pr}\label{becares}
 Assume that $S \in %Obj 
\obj\underline{Ht}_{hom}^{eff} \subset \mathfrak{D}^{eff}$. Then the following conditions are equivalent:\\
1. $S \in \mathfrak{D}^{eff}_{w_{Smooth}^{eff, s} \geq0}$.\\
2. Condition 2 %-- (4) 
 of  Proposition \ref{epowbo} is fulfilled if one takes $r$ to be the minimal integer $m\ge -1$ such that  $s_{m+1} \geq -m-1$ if $m$ of this sort exists and 
 %for the minimal $r\ge 0$ such that $s_{r+1} \geq 0$ for $j > r$; $r=-1$ if $s_{j}<0$, finite and 
  $r=+\infty$ otherwise (see Remark \ref{eot}(2)).
\end{pr}

\begin{proof}
1 $\Rightarrow$ 2. This implication follows from the definition of $w_{Smooth}^{eff, s}$, Proposition \ref{thshc}(5) and Remark \ref{oprnci}.\\
2 $\Rightarrow$ 1. Condition (2) implies that $\mathfrak{D}^{eff}(\mathcal{M}(X), S[i]) = \ns$ for $i> r$. This implies that $S$ belongs to the class $\mathcal{A}_{s} $, where we choose the sequence $s$ as in the formulation. Namely, we set 
$s_{\geq r + 1} = - r-1$ if the first of the alternatives in (2) holds, and $(s_j) = -\infty$ otherwise (recall that we assume all our sequences to be non-decreasing). 
\end{proof}

\begin{rema}
1. Clearly, in the case $r<+\infty$ conditions %3 and 4 
 3--5 of  Proposition \ref{epowbo} are equivalent to our ones as well.

2. %We also %note that the main 
Actually, our principal examples are %obtained for
 %the ones for $s$ as in
 $s=s^{r-bir}$ (see Proposition \ref{epowbo}) % (that gives the $r-$birational subcategory) 
   and $s=0$ (that corresponds to the smooth weight structure $w_{Smooth}^{eff}$).
\end{rema}

Now we define some new $t$-structures.

\begin{defi}
\begin{enumerate}
\item We will say that a $t-$structure $t$ is {\it{right adjacent}} to $w$ whenever $\cu_{w\geq 0}=\cu_{t\geq 0}$.
\item If $H: \cu \to \au$ is contravariant functor, $\au$ is an abelian category, then we define $W^{j}(H)(X):= \imm(H(w_{\geq j}(X)) \to H(X))$ (see Remark \ref{rstws}). The map thus defined gives a canonical subfunctor of $H$ (in particular, this image does not depend on the choice of a weight decomposition of $X[j]$); see Proposition 2.1.2 of \cite{bws}.
\end{enumerate}
\end{defi}
We will omit the adjective ''right'' in this definition and simply write "adjacent", since further we will %consider only
 need  this case of the definition only. Recall also that $\mathfrak{D}^{eff}$ satisfies the Brown representability property since it is compactly generated (see Ch. 8 %Proposition 8.4.2  
  of \cite{neebook}).

\begin{defnc}\label{codfnogf}
\begin{enumerate}
\item Since all the weight structures $w_{Smooth}^{eff, s}$ are smashing (see Definition \ref{dwso}(\ref{smash})), there exists a $t-$structure $t_{Smooth}^{eff, s}$ adjacent to it;  see Theorem 3.2.3(I) of \cite{bvtr}.
Thus we have a canonical map $\tau^{t_{Smooth}^{eff, s}}_{\geq -j}(E) \to E$ for any $E \in \mathfrak{D}^{eff}$; see Definition \ref{dtst}(iii) and Remark \ref{ronts}(1). 
Further, we can apply $H^{t^{eff}_{hom}}_{0}$ to this map, and get  a canonical morphism $H^{t^{eff}_{hom}}_{0}(\tau^{t_{Smooth}^{eff, s}}_{\geq -j}(E) \to E)$. 
As above, we will write just $t^{eff}_{Smooth}$ in the case $s_{j}=0$ (for ${j \ge 0}$).
\item %That is, i
If we apply the previous construction to $E \in \obj\underline{Ht}^{eff}_{hom}$ we get maps from $\underline{Ht}^{eff}_{hom}\cap \mathfrak{D}^{eff}_{t_{\smo} \geq -j}$ to $\underline{Ht}^{eff}_{hom}$. They give rise to the following filtration: %on $\underline{Ht}^{eff}_{hom}$. We set 
 $F^{j, s}(E)= \imm_{\underline{Ht}^{eff}_{hom}} (H^{t_{hom}}_{0}(\tau^{t_{Smooth}^{eff, s}}_{\geq -j}(E)) \to H^{t_{hom}}_{0}(E)=E)$.
\end{enumerate}
\end{defnc}

Now we will study in more detail our filtration for the $t-$structure, generated by $\mathcal{A}_{s}$. %as in Proposition \ref{epowbo}.

\begin{theo}\label{wbirssh}
Let $-1\le i\le +\infty$, $j\in \z$.
\begin{enumerate} 
\item The category $Ht^{i-bir, s}$ of $i-$birational objects is a Serre abelian subcategory of $\underline{Ht}^{eff}_{hom}$, and we denote by $j_{i}$ this inclusion.

Fix $E \in\obj \underline{Ht}^{eff}_{hom}$.

\item %For $E$ as above, 
The object $F^{j, s}(E)$ is the maximal $r$-birational subobject of $E$ in ${\hrt}^{eff}_{hom}$, where we take $r$ to be the minimal $m\ge -1$ such that  $s_{m+1} \geq -j$ if $m$ of this sort exists and  $r=+\infty$ otherwise (see Remark \ref{eot}(2)).

Thus, $F^{j, s}$ is  right adjoint to the inclusion $j_{r}$.

\item The {\it higher unramified part} functor $R_{nr, i}$ (see \S\ref{wocvs}; cf. also \S7.1 of \cite{kabir}) gives the 
$t-$truncation $\tau^{t_{Smooth}^{eff, s^{i-bir}}}_{\geq 0}.$ 
\item Similarly, $w_{Smooth}^{eff, s^{i-bir}}$-truncations 
give the corresponding slices (see the beginning of \S\ref{wocvs}).
\item $E$ belongs to $\underline{Ht}^{eff}_{\smo}$  if and only if it is birational (see Remark \ref{eot}(1)).

\end{enumerate}
\end{theo}

\begin{proof}
1. Suppose that $0 \to S' \to S \to S'' \to 0$ is an exact sequence in $\hrt^{eff}_{hom}$. 
Since the stalks as in condition (4) of Proposition \ref{epowbo} give exact functors $\hrt^{eff}_{hom}\to \ab$ (see our Corollary \ref{conspecseq}(\ref{axcor1})), 
$S$ belongs to $\hrt^{i-bir, s}$ if and only if $S'$ and $ S''$ do. 
 
2. The sheaf $F^{j, s}(E)$ is $r-$birational by Propositions \ref{becares} , \ref{wcatheq} and assertion (1) above. Let $\widetilde E \stackrel{f}{\hookrightarrow} E$ be  some $r$-birational subobject of $E$.   
Since $\widetilde E\in \mathfrak{D}^{eff}_{t_{Smooth \geq -j}^{eff, s}}$,
the adjunction corresponding to $\tau^{t_{Smooth}^{eff, s}}_{\geq -j}$  (see Remark \ref{ronts}(1)) yields that  $f$ factors through $\tau^{t_{Smooth}^{eff, s}}_{\geq -j}(E)$.  Applying $H_{0}^{\hom}$ to the correspondiing commutative triangle we obtain that $f$ also factors through $F^{j, s}(E)$.

3. Our assertion follows from the adjunctions $\mathfrak{D}^{eff}_{t_{Smooth}^{eff, s^{i-bir}} \geq 0} \rightleftarrows \mathfrak{D}^{eff}$ and $i^{(i)} \dashv R_{nr, i}$.

4. Note that our weight structure is also a $t-$structure (by definition). Thus, our statement is an obvious consequence of the definition of slices and corresponding adjunctions (cf. assertion (3) above and \S\ref{wocvs}).

5. Obvious; note  that $\mathfrak{D}^{eff}_{{t}_{hom} \leq 0} \subset \mathfrak{D}^{eff}_{{t}_{\smo} \leq 0}$.
\end{proof}

\begin{rema}
1. In particular, for $E \in \obj \hrt^{eff}_{hom}$ we have $ F^{j}(E) = E$ if and only if $E$ is $j-$birational.

2. It would be interesting to relate our filtration with the filtration from Theorem 15 of $\cite{{bac}}$ (in the case $\mathfrak{D}^{eff}=SH(k)^{eff}$).  %of $Ht^{SH}_{hom}$).

3. %Let's take 
 For the weight structure as in assertion (4) above 
  $W^{*}(\mathfrak{D}^{eff}(-, R(n)[m]))$ gives the filtration from Definition 5.2.1 and Corollary 5.3.3 of \cite{pel1}. Note that for $m=2n$ 
and $\mathfrak{D}^{eff}=DM^{eff}$ this is a candidate for a Bloch-Beilinson-Murre filtration (see Theorem 6.1.4, Conjecture 6.1.7, Proposition 6.1.8 and Rem. 6.1.9 of ibid.).
\end{rema}

Next we apply our results to the study of unramified cohomology.

\begin{coro}\label{spobc}
1. If $S \in \obj \hrt^{eff}_{hom}$ then 
the %counit of the adjunction  
 $\hrt^{eff}_{hom}$-monomorphism  $c(S): F^{0}(S) \to S$ from Construction-Definition \ref{codfnogf} gives the {\it unramified part} of $\mathfrak{D}^{eff}_{R}(-, S)$ in the sense of \cite[Definition 7.2.1]{kabir}.

2. %Under the conditions of the previous assertion, let 
 Let $M \in \underline{Hw}_{Smooth^{eff}_{R}}$ and for some $X\in\sv$ and $f: \mathcal{M}_{R}(X) \to M$ assume that $p_{0}(f)$ is an isomorphism.
Then %the morphism 
$\mathfrak{D}^{eff}_{R}(f,S)$ is %injective
 monomorphic, and its image yields the unramified part of $\mathfrak{D}^{eff}_{R}(\mathcal{M}_{R}(X), S)$.

\end{coro}

\begin{proof}

1. %The argument is similar to that for???
 See Proposition 2.6.3 and Theorem 7.3.1 of \cite{kabir}, and \cite[Lemma 4.2]{aso}  (note that the proofs given there are applicable to any of the motivic categories from \S\ref{vgabc})

2. We argue similarly to \cite[Theorem 2.2.3]{bcpe}. Consider the following commutative diagram

\[
\begin{tikzcd}[row sep=huge, column sep=huge]
\mathfrak{D}^{eff}_{R}(M, F^{0}(S)) \arrow{r}{\cong} \arrow[swap]{d}{c(S)_{*}} & \mathfrak{D}^{eff}_{R}(\mathcal{M}_{R}(X), F^{0}(S)) \arrow{d} \\
\mathfrak{D}^{eff}_{R}(M, S)  \arrow{r} &\mathfrak{D}^{eff}_{R}(\mathcal{M}_{R}(X), S)
\end{tikzcd}
\]

Firstly, by adjunctions and Theorem \ref{wbirssh} we have $\mathfrak{D}^{eff}_{R}(i^{(0)}p_{0}(M), S) \cong \mathfrak{D}^{eff}_{R}(M, i^{(0)}R_{nr, 0}(S)) \cong \mathfrak{D}^{eff}_{R}(M, F^{0}(S))$.
Now, let us look at the following exact sequence coming from the slice filtration triangle (\ref{sls}): 
$$0 \to \mathfrak{D}^{eff}_{R}(M, F^{0}(S)) \to\mathfrak{D}^{eff}_{R}(M, S) \to \mathfrak{D}^{eff}_{R}(\nu^{\geq 1}(M), S)$$
It remains to prove that the map $c(S)_{*}$ is epimorphic. This fact follows from the sequence above along with the orthogonality axiom for $t-$structures,  since %note that 
 $\nu^{\geq 1}(M) \in \mathfrak{D}^{eff}_{t_{hom} \geq 1}$ (see Lemma 2.2.4(2) of \cite{bokum} and \cite[Lemma 6.1(2)]{by}).
\end{proof}

\begin{pr}\label{vopraadjw}

1. Let $E, X \in \mathfrak{D}^{eff}$. Then $W^{i}(\mathfrak{D}^{eff}(-, E))(X) \cong \imm(\mathfrak{D}^{eff}(X, \tau^{t_{Smooth}^{eff, s}}_{\geq -i}(E)) \to \mathfrak{D}^{eff}(X, E))$.

2. For all $i,j$ and $X,Y \in \mathfrak{D}^{eff}$ we have a functorial isomorphism
\scriptsize $${\mathfrak{D}^{eff}(X, \tau^{t_{Smooth}^{eff, s}}_{\geq -i}(Y)[j+i]) \cong \imm(\mathfrak{D}^{eff}(w_{Smooth \geq j}^{eff, s}(X), Y[i]) \to \mathfrak{D}^{eff}(w_{Smooth \geq j-1}^{eff, s}(X), Y[i+1])).}$$\normalsize

3. $X \in \mathfrak{D}_{w_{Smooth}^{s}=i} \Longrightarrow \forall Y \in \mathfrak{D}$ we have $\mathfrak{D}(X,Y) \cong \mathfrak{D}(X,H_{i}^{t_{\smo}^{s}}(Y))$.

4. $\mathfrak{D}^{+}_{\wchosw_{\geq 0}}=\obj(\mathfrak{D}^{+})\bigcap ^{\perp}(\bigcup_{i<0}\mathfrak{D}_{t_{\smo}=i})$ (see Definition \ref{dwso}(\ref{lrbo}) for the first piece of this notation).

5. If $C \in \mathfrak{D}_{{t}_{hom} \leq i}$, $U \in \sv,\ i \geq 0$ then %\small 
 $$\mathfrak{D}(\mathcal{M}(U),\tau^{t_{\smo}}_{\geq -i-1}(C[-i-1])) \cong \mathfrak{D}(\wchosw_{\geq -i}\mathcal{M}(U), C).$$ % for $U \in \sv, i \geq 0.$$ %\normalsize
\end{pr}
\begin{proof}

1, 2. See Theorem 4.4.2(6,7) of \cite{bws}.

3. See Theorem 2.6.1(4) of \cite{bger}.

4. Obviously, the first of these classes is contained in the second one. Let $X$ be an object of $\mathfrak{D}^{+}$. To check the inverse inclusion, 
it suffices to prove that $X \perp Y$ for $X \perp (\bigcup_{i<0}\mathfrak{D}_{t_{\smo}=i})$, and for all $Y \in \mathfrak{D}_{t_{\smo}\leq -1}$.
Since $X$ is $\wchosw$-bounded below, we have $X \perp \tau^{t_{\smo}}_{\leq i}(Y)$ for some $i$. Thus, considering appropriate $t-$decompositions for $Y$ and its truncations, we obtain the  statement in question.

5. Take the weight decompositions as in  Proposition \ref{pbw}(\ref{icompl}) for $m=-i-2, l=-i-1, g=\id_{\mathcal{M}(U)}$. %Further, applying
Applying the functor $\mathfrak{D}(-, C)$ one can check that the map 
$\mathfrak{D}(\wchosw_{\geq -i}\mathcal{M}(U), C) \to \mathfrak{D}(\wchosw_{\geq -i-1}\mathcal{M}(U), C)$ is injective. Now the claim follows  from Theorem 4.4.2(7) of \cite{bws} immediately.
\end{proof}

\subsection{More on $n$-birational categories and weight-exact localisations}\label{etfiloth}

In this section we prove some weak weight-exactness properties of the functor $-\lan n \ra$. As a consequence, we get the $w^{eff}_{Smooth}$ and $w_{Smooth}-$exactness of this functor, as well as of the localization functor $p_{n}$.
Note that in this paragraph, we exclude the case of $SH^{S^{1}}(k)$ (cf. Remark \ref{cgmus}(3,4)). %; the main reason for this is the lack of the Voevodsky's cancellation theorem for the $SH^{S^{1}}(k)$).

\begin{pr}\label{wawwes}
Assume that $k, n\ge 0$ are fixed and  $s_{j+n}-s_{j} \leq k$ for $j\ge 0$;  take   $C \in \mathfrak{D}^{eff}_{w_{Smooth}^{eff, s}\geq 0}$ for $s=(s_{j})$. %\}_{j \in \mathbb{N}}$.   
 Then %we have
   $C\lan n \ra \in \mathfrak{D}^{eff}_{w_{Smooth}^{eff, s}\geq -k}$.

\end{pr}

\begin{proof}
 Let $C\in \mathfrak{D}^{eff}_{w_{Smooth}^{eff, s}\geq 0}$, and consider the following %family of
   functors from the category $\sv$ of smooth $k$-varieties: $H^{q}_{p}(X):=\mathfrak{D}^{eff}(\mathcal{M}(X)\{p\}[-q], C\lan n \ra[-k])$ for
each $p \geq 0, q \in \mathbb{Z}, X \in \sv$. We should verify that $H^{p+q}_{j}(X) = \ns$ for all $q >0,\ X \in \sv$. Let us write the coniveau spectral sequence as in %the 
 \textbf{(A1)}:
$$E_{1}^{p,q}=\underset{x\in X^{(p)}}{\coprod}H^{q}_{j+p}(x) \Rightarrow H^{p+q}_{j}(\mathcal{M}(X)).$$
Then $E_{1}^{0,q}=\ns$ for $q \geq 0$ by \textbf{(A2)}. 
Further, for a function field $K/k$ and $p >0$ it is easily seen that  $H^{q}_{p+j}(K)= \mathfrak{D}^{eff}(\mathcal{M}(K)\lan j+p-n \ra[k-j-p-q], C)= \ns$ % from easy computation and our conditions on $s_{j}$ 
 (see Definition \ref{thmdf}(4) for this notation; cf. also the proof of Proposition \ref{amofgw}). 
It obviously follows that $H^{p+q}_{j}(X) = \ns$ if $p+q >0$, and that is what we need.

\end{proof}

\begin{rema}
Clearly, a similar property holds in the non-effective case (cf. Definition \ref{naortcon}).
\end{rema}

\begin{coro}\label{correct}
1. The endofunctors $-\lan n \ra$ are $w_{Smooth}^{eff}-$exact (resp. $w_{Smooth}-$exact) for $n \geq 0$ (resp. $n\in \mathbb{Z}$).

2. The embedding $i: \mathfrak{D}^{eff} \hookrightarrow \mathfrak{D}$ is weight-exact with respect to $\wchose$ and $w_{Smooth}$.

3. Take a sequence $(s_{j})_{j\ge 0}$ 
 such that $s_{j+n}-s_{j} \leq 1$ for all  $j\ge 0$. Then the functor $p_{n}$ is weight-exact with respect to the weight structures $w_{Smooth}^{eff, s}$ and $w^{n-bir, s}$,
where the weight structure $w^{n-bir, s}$ is defined %the same formula  
 via $\mathcal{M}^{n-bir}_{R}$ similarly to Definitions \ref{tgwss} and \ref{naortcon}. 

4. $\mathfrak{D}_{\wchose \geq 0}^{eff} \subset \mathfrak{D}^{eff}_{t_{hom}^{eff}\geq 0}$.

\end{coro}

\begin{proof}
1. Take $s_{j}=0$ for %${j \in \mathbb{Z}}$ 
 $j\ge 0$ (resp. $j\in \mathbb{Z}$) in the previous proposition.

2. From Proposition \ref{pcgws}(\ref{pc2}) we obtain that $i$ is left weight-exact. To prove the right weight-exactness, we take $C\in \mathfrak{D}_{\wchose \geq 0}^{eff}$, and consider the group 
$\mathfrak{D}(\mathcal{M}_{R}(X)\lan n \ra[s], C)$, where $X \in \sv, n \in \mathbb{Z}, s<0$. Then in the case of $n \geq 0$ this group vanishes by the previous assertion. If $n < 0$, then %we have the following 
$\mathfrak{D}(\mathcal{M}_{R}(X)\lan n \ra[s], C) \cong \mathfrak{D}^{eff}(\mathcal{M}_{R}(X)[s], C\lan -n \ra)$, and $C\lan -n \ra \in \mathfrak{D}_{\wchose \geq 0}^{eff}$, and this group vanishes also. Thus, we are done.

3. Take a weight decomposition $X \to C \to Y \to X[1]$ with respect to $w_{Smooth}^{eff, s}$, and twist it by $-\lan n \ra$; then we have $Y\lan n \ra \in \mathfrak{D}^{eff}_{w_{Smooth}^{eff, s} \geq 0}\cap \mathfrak{D}^{eff}\lan n \ra$ by Proposition \ref{wawwes}, 
and $X\lan n \ra \in \mathfrak{D}^{eff}_{w_{Smooth}^{eff, s} \leq 0}\cap \mathfrak{D}^{eff}\lan n \ra$ by definition. Thus our statement follows from  Theorem 3.1.3(2) of \cite{bsnew}.

4. Immediately from definition and Proposition \ref{thshc}(1) (see also Lemma 2.2.4(3) of \cite{bokum}).
\end{proof}

%\begin{coro}
%In the category $\mathfrak{D}^{n-bir}$ we have 
 %the following inclusions:
%  For every $m\geq n\in \z$ we have the following: \\
%(i) $\mathfrak{D}^{n-bir}_{{t}_{\smo} \geq 0} \subset \mathfrak{D}^{n-bir}_{{t}_{hom} \geq 0}$\\
%(ii) $\mathfrak{D}^{n-bir}_{{t}_{\smo} \geq 0}\{m\} \supset \mathfrak{D}^{n-bir}_{{t}_{hom} \geq 0}$. %, $m\geq n$.
%\end{coro}

%\begin{proof}
%Firstly, we note that $p_{n}$ is right ${t}_{\smo}$-exact by Proposition \ref{spobc}(1) and Proposition 4.3.1.3.II(1) of \cite{bos}. Thus the first inclusion is obvious, and the second one immediately follows from our description of $t_{hom}$ in \S\ref{thomaxc}. %2.
%\end{proof}

\section{Our constructions in motivic categories distinct from $SH$}\label{sother}

For the convenience of the reader, now we discuss our main constructions and results in the case where $\mathfrak{D}$ is distinct from $SH$. As noted earlier, these results are quite similar to those obtained in Proposition \ref{thshc}.

\subsection{%Applications toMore detail for
Our notions and %results in other
 assumptions in various motivic categories}\label{sothermot}

In this section, we check our axioms and discuss the related definitions for the motivic categories %, other than 
 distinct from $SH^{eff}(k) \subset SH(k)$. %???In particular, note that in the categories listed below, axioms \textbf{(A1)}-\textbf{(A2)} are fulfilled (cf. the results from \S1.4, \S4 and \S5.1 of \cite{bgn}, and Lemma \ref{mlfjooa} above). 
%Note also

 Firstly, recall that there exists a canonical exact monoidal connecting functor%s
   from $SH^{S^{1}}(k)$ into %all of 
 each  our motivic category $\mathfrak{D}(k)$; it send $\mathcal{M}_{SH^{S^{1}}}(X)$ into the corresponding $\mathcal{M}_{\mathfrak{D}}(X)$. %, and these functors  and "commute with twists". 
  This %statement 
   gives the isomorphism  $\mathcal{M}_{\mathfrak{D}}(\mathbb{P}^{1}) \cong \ood \oplus T$ mentioned in the beginning of \S\ref{vgabc}. 
%Thus one has 
 %Next, one obtains
 Thus the spectral sequences of  axiom \textbf{(A1)} in all of these $\mathfrak{D}(k)$ can be obtained from the one for $SH^{S^{1}}(k)$.
%Furthermore, 

 Now %we  check
  let us check  axiom \textbf{(A2)} %in each case 
  for each of our categories separately; cf.  Lemma \ref{mlfjooa} above.\\ 

% More precisely, for any motivic category below we have an adjunction with $SH^{S^{1}}(k)$, and thus \textbf{(A2)} can be derived from Lemma 4.2.7 of \cite{Mor} and Theorem 6.1.8 of \cite{Mor1}.
%Further, $t-$exactness of Voevodsky's contraction can be obtained from Remark 4.3.12 of \cite{Mor}, and nondegeneracy from 5.1.14 of \cite{Mor} (cf. also Corollary 2.4 of \cite{hoy}).\\

\fbox{The case of $SH^{S^{1}}(k)$}\\

%The main difference between this case and the  remaining ones is the non-invertibility of the twist functor.\\
We will write $\pi^{\mathbb{A}^{1}}_{n}(E)$ for  the Nisnevich sheaf  associated to the presheaf $X \mapsto E^{-n}_{0}(X)=SH^{S^{1}}(\mathcal{M}(X_{+})[n], E), X \in \sv, n \in \mathbb{Z}$.

\begin{pr}\label{thsss}
\begin{enumerate}
\item $E \in SH^{S^{1}}(k)_{\geq 0}$ if and only if $\pi^{\mathbb{A}^{1}}_{n}(E)=0$ for $n<0$.
\item $E \in SH^{S^{1}}(k)_{\leq 0}$ if and only if $\pi^{\mathbb{A}^{1}}_{n}(E)=0$ for $n>0$.
\item The functor $\pi^{\mathbb{A}^{1}}_{0}(E)$ gives an equivalence of $\underline{Ht}^{SH^{S^{1}}}_{hom}$ and $SHI(k)$, where $SHI(k)$ is the category of strictly homotopy invariant sheaves (see Definition 4.3.5 of \cite{Mor}).
\item \textbf{(A2)} is fulfilled.
%\item The endofunctor $(-)_{-i}$ is $t_{hom}-$exact.
%\item The functor $\Phi^{SH^{S^{1}}}: \underline{Ht} \to Psh(Pts, Ab)$ defined as follows $\Phi^{SH^{S^{1}}}(F)(\spe(K))=SH^{S^{1}}(M(\spe(K)), F)$ for all function fields $K/k$ 
%is conservative.

\end{enumerate}
\end{pr}

\begin{proof}
1,2. Similarly to Theorem 2.3 of \cite{hoy}.\\
3. See Lemma 4.3.7(2) of \cite{Mor}.\\
4. See Lemma 6.1.6 of \cite{Mor1}.
%See Corollary \ref{conspecseq}(\ref{axcor0}), cf. also Remark 4.3.12 of \cite{Mor}.\\
%5. See Corollary \ref{conspecseq}(\ref{axcor1}). %cf. also Our statement follows from Example 2.3 and Proposition 2.8 of \cite{Mor2} (see also Corollary 4.3 of \cite{aso}).
%Firstly, by (4) we can remove $i$ in the hom-set. Further, note that our functor is the cohomology theory in the sense of Chapter 5 of \cite{chk}. It is well known that the axioms COH1, COH3 of loc.cit. holds for this theory.
%Thus, by Theorem 6.2.1 of ibid., we have conservativity for semi-local schemes. Finally, applying an Example 2.3 of \cite{Mor2}, we obtain our statement.
\end{proof}

\fbox{The case of $D_{\mathbb{A}^{1}}(k)$ and $D_{\mathbb{A}^{1}}^{eff}(k)$}

\begin{pr}\label{thdera}
\begin{enumerate}
\item $E \in D_{\mathbb{A}^{1}}(k)_{\geq 0}$ if and only if $\underline{H}^{\mathbb{A}^{1}}_{m,n}(E)=0$ for $(m,n) \in \mathbb{Z}\times\mathbb{Z}^{-}$.
\item $E \in D_{\mathbb{A}^{1}}(k)_{\leq 0}$ if and only if $\underline{H}^{\mathbb{A}^{1}}_{m,n}(E)=0$ for $(m,n) \in \mathbb{Z}\times\mathbb{Z}^{+}$.
\item There exist equivalences $\underline{H}^{\mathbb{A}^{1}}_{*,0}: \hrt^{D_{\mathbb{A}^{1}}}_{hom} \to HI_{*}(k)$ and $\underline{H}^{\mathbb{A}^{1}}_{0,0}: \hrt^{D_{\mathbb{A}^{1}}^{eff}} \to HI^{fr}(k)$.
\item \textbf{(A2)} is fulfilled.
%The functor $-\{i\}$ is $t_{hom}-$exact for any  $i\in \mathbb{Z}$.
%\item The functor $\Phi^{D_{\mathbb{A}^{1}}}: \underline{Ht} \to Psh(Pts, Ab^\mathbb{Z})$ defined as follows $\Phi^{D_{\mathbb{A}^{1}}}(F)(\spe(K))=D_{\mathbb{A}^{1}}(M(\spe(K))\{i\}, F)$ for all $i \in \mathbb{Z}$
%(resp. $\Phi^{D_{\mathbb{A}^{1}}^{eff}}_{0}(F)(\spe(K))=D_{\mathbb{A}^{1}}^{eff}(M(\spe(K)), F)$) for all function fields $K/k$ is conservative.
\end{enumerate}
\end{pr}

\begin{proof}
1, 2. See subsection 16.2.4 of \cite{cidtm} or Corollary 2.1.72 of \cite{ayo}.\\
3. See Example 4.1.2 of \cite{bondegl}, Theorem 8.12 and Corollary 8.14 of \cite{ane}.\\
4. Similarly to the previous case; see Remark 8 of \cite{Mor1}.
% It is sufficient to prove this for $i \geq{0}$, and this case is contained in Theorem 2.2.51 of \cite{ayo} (Actually, only the case of  rational coefficients is considered there but the proof immediately extends to the general case).\\
%5. See Corollary \ref{conspecseq}(\ref{axcor1}). %See Proposition 4.2.2 of \cite{bondegl}.
\end{proof}

\fbox{The case of $DM(k)$ and $DM^{eff}(k)$}

\begin{pr}\label{thpragg}
\begin{enumerate}
%\item The $t-$structures $t_{hom}^{DM}$ and $t_{hom}^{{DM}^{eff}}$ are non-degenerate.
\item %The following equivalences exist: 
 There exist equivalences $\underline{H}_{*}^{0}: \hrt^{DM} \to HI_{*}^{tr}(k)$ and $\underline{H^{0}}: \hrt^{DM^{eff}} \to HI(k)$, where $HI_{*}^{tr}(k)$ and $HI(k)$  are the categories of homotopy modules with transfers and homotopy invariant sheaves with transfers, respectively (see \cite[Definition 3.1.9]{vo} and \cite[Definition 1.3.2]{degmod}).
%\item $H_{0}^{t_{hom}}(C\{n\})=H_{0}^{t_{hom}}(C)\{n\}$ for all $C\in \mathfrak{D}, n \in \mathbb{Z}$. Also the  functor $(-)_{-1}$ is $t_{hom}-$exact.
%\item The functor $\Phi^{DM}: \underline{Ht} \to Psh(Pts, Ab^\mathbb{Z})$ defined as follows $\Phi^{DM}(F)(\spe(K)) = DM(M(\spe(K))\{-n\}, F)$ for all $n \in \mathbb{Z}$
%(resp. $\Phi_{0}(F)(\spe(K)) = DM^{eff}(M(\spe(K)), F)$ ; see Definition \ref{thmdf}(4)) for all function fields $K/k$, is conservative.
%\item $t_{hom}^{{DM}^{eff}}$ can also be described as follows: $DM^{eff}_{t_{hom} \leq 0}= \{C \in DM^{eff} | DM^{eff}(\mathcal{M}(K)[i], C)=0$ for any function field $K/k$, $i>0\}$.
%The formula in brackets is just a notational shorthand for the functor on $\mathfrak{D}^{eff}: C \mapsto\varinjlim\limits_{X, k(X)=K}\mathfrak{D}^{eff}(\mathcal{M}(X)[i], C)$.\\
%\item There is an adjunction $ i^{DM}: DM^{eff}(k) \rightleftarrows DM(k): \omega^{DM}$. Moreover,  $i^{DM}$ is right $t_{hom}-$exact, and $\omega^{DM}$ is  $t_{hom}-$exact.
\item \textbf{(A2)} is fulfilled.
\end{enumerate}
\end{pr}

\begin{proof}
%1. See Corollary \ref{conspecseq}(\ref{axcor2}), cf. also Lemma 5.5(2) of \cite{degmod}.\\ %\S4.3, \S5.2 of \cite{Mor} and
1. See \cite[Proposition 3.1.12]{vo} and \cite[Corollary 5.2 and Theorem 5.11]{degmod}.\\
%3. See Corollary \ref{conspecseq}(\ref{axcor0}), cf. also section 5.7(2) of \cite{degmod}.\\ 
%4. See Corollary \ref{conspecseq}(\ref{axcor1}).\\ %Our arguments is similar to the [\cite{bondegl}, Theorem 3.3.1]. Thus, consider the following spectral sequence (see Proposition 4.3.1 of \cite{bger}): 
%$$E_{1}^{p,q}=\underset{x\in X^{(p)}}{\coprod}F^{q}_{p+i}(x) \Rightarrow F^{p+q}_{i}(X),$$ where $X^{p}$ is the set of points 
%of $X$ of codimension $p$, $i \in \mathbb{Z}$. Take $F \in Ht$, such that $\Phi^{DM}(F)=\ns$. We should verify that $F^{0}_{i}(X)=\ns$ for all $X \in \sv, i \in \mathbb{Z}$.
%Now, if $p=q=0$, then $F^{0}_{i}(x)=\ns$ by our assumptions, and if $p>0$, then $F^{-p}_{p+i}(x)=\ns$ by $t-$exactness of $-\{p+i\}$ and orthogonality axiom for $t$-structures. 
%This yields our assertion.\\
%5. Easily %, by using 
% via the coniveau spectral sequence (see the proof of Proposition \ref{thshc}(5)).\\
   %See Proposition 4.2.2 of \cite{bondegl} (cf. also Theorem 5.2.3(6) of \cite{bgern}).\\
%2. See Proposition 2.2.5(1) of \cite{binfeff}.\\
2. This is well known; see \S5 of \cite{degmod} (cf. also Lemma 2.2.4 of \cite{bokum}).
\end{proof}

\fbox{The case of $\mgl-\modd(k)$}

\begin{pr}\label{thonmod}
\begin{enumerate}
\item There exists an equivalence $\mathcal{O}|_{\hrt}: \hrt^{\mgl} \cong HI_{*}^{tr}(k)$.
\item \textbf{(A2)} is fulfilled.
%\item The functor $\Phi^{\mgl-\modd(k)}: \underline{\hrt}^{\mgl} \to Psh(Pts, Ab^\mathbb{Z})$ defined as follows $\Phi^{\mgl-\modd(k)}(F)(\spe(K))= \mgl-\modd(k)(M(\spe(K))\{i\}, F)$ for all $i \in \mathbb{Z}$ and all function fields $K/k$ is conservative.

\end{enumerate}
\end{pr}
\begin{proof}
1. See Remark 4.3.3 of \cite{bondegl}.\\
2. Easy from the corresponding result for $SH(k)$, $t_{hom}-$exactness and conservativity of the ''forgetful'' functor $\mathcal{O}$, which is right adjoint to the natural connecting functor $\mathcal{L}: SH(k) \to \mgl-\modd(k)$ (for which
 $\mathcal{L}(\mathcal{M}_{SH}(X)\{i\}[n])= \mgl(X)\{i\}[n]$);
see Example 2.3.3 of \cite{bondegl} and \S2.2.5 of \cite{degormod}.
%See Corollary \ref{conspecseq}(\ref{axcor1}).

\end{proof}

\begin{rema}\label{oprnci}
Recall also that for each of the categories $\mathfrak{D}^{eff}(k)$ there is a ''forgetful'' $t^{eff}_{hom}-$exact functor $\Psi_{\mathfrak{D}}$ to $SH^{S^{1}}(k)$, which is right adjoint to the natural functor $SH^{S^{1}}(k)\to \mathfrak{D}^{eff}$
(easily follows from our description of $t^{eff}_{hom} $ in terms of homotopy sheaves, see also Lemma 6.2 (2) of \cite{by}).

Let $E\in \underline{Ht}^{\mathfrak{D}^{eff}}_{hom}$; then $E^{n}_{0}(X) \cong H_{Nis}^{n}(X, \pi^{\mathbb{A}^{1}}_{0}(\Psi_{\mathfrak{D}}(E)))$ for all $X \in \sv$, $n \in \mathbb{Z}$. Indeed, in the case $\mathfrak{D}^{eff}= SH^{S^{1}}(k)$ this statement
 is well known and is given by Proposition 1.4.6(6) of \cite{bgn}. The general case %is obtained
  follows  from this one via the adjunction isomorphism $E^{n}_{0}(X) \cong \Psi_{\mathfrak{D}}(E)^n_0(X)$; cf. Proposition \ref{thshc}(5).
\end{rema}

\subsection{On localizations of coefficients in motivic categories}\label{olctc}

%In this section 
 Now we recall some %notations and elementary facts 
  basics on localizations of  coefficient rings (actually, we  "start from" the case where the coefficient ring is just $\z$)   in  compactly generated triangulated categories. 

Below $S \subset \z$ will be a set of prime numbers; denote the ring $\mathbb{Z}[S^{-1}]$ by $R$.

\begin{pr}\label{cofloc}

Assume that $\cu$ is compactly generated by a small subcategory $\cu'$. Denote by $\cu_{S-tors}$ the localizing subcategory of $\cu$ (compactly) generated by $\co(c' \xrightarrow{\times s} c')$ for $c' \in \obj\cu'$, $s \in S$.
Then the following statements are valid.

\begin{enumerate}
\item $\cu_{S-tors}$ %also 
 contains the cones of $c \xrightarrow{\times s} c$ for all $c \in \obj\cu$, $s \in S$.

\item The Verdier quotient $\cu_{R}=\cu/\cu_{S-tors}$ exists (i.e., morphisms form sets); the localization functor $l: \cu \to \cu_{R}$ respects coproducts and converts compact objects into compact ones.
Moreover,  $\cu_{R}$ is generated by $l(\obj\cu')$ as a localizing subcategory.

\item For any $c \in \obj\cu$, $c' \in \obj\cu'$, we have $\cu_{R}(l(c'), l(c)) \cong \cu(c', c)\otimes_{\mathbb{Z}} R$.

\item $l$ possesses a right adjoint $G$ that is a full embedding functor. The essential image of $G$ consists of those $M \in \obj\cu$, such that $s Id_{M}$ is an automorphism for any $s \in S$.

\item Assume $\cha(k)=p$, and $p \in S$ if $p>0$. Then the category $\mathfrak{D}_{R}^{c}(k)$ is rigid. Moreover, $\mathfrak{D}_{R}^{c}(k)$ is the smallest thick subcategory of $\mathfrak{D}_{R}(k)$ containing
all $\mathcal{M}_{R}(P)\{i\}$ for $P \in \spv$, $i \in \mathbb{Z}$. Thus, the set $\mathcal{M}_{R}(\spv)$ compactly generates $\mathfrak{D}_{R}(k)$.

Surely, the corresponding statements are fulfilled for the effective case.

\end{enumerate}
\end{pr}

\begin{proof}
1, 2, 3, 4. See Proposition 1.2.5 of \cite{binfeff}. %(cf. also Proposition 5.6.2(I) of \cite{bpgws}).

5. See \cite[Proposition 2.2.3(9)]{binfeff} (cf. also Corollary 2.4.8 of \cite{bondegl} and \cite[Lemma 2.3.1]{bzp}).
\end{proof}

\begin{rema}
\begin{enumerate}
\item Since $\chow_{R}$ is Karoubian and connective in $DM_{R}$ (i.e., $\chow_{R} \perp \chow_{R}[i]$ for all $i>0$; see Remark 3.1.2 of \cite{bokum}), Corollary 2.1.2 of \cite{bonspkar} gives a unique weight structure $w_{\chow,gm}$ on the smallest full 
retraction-closed triangulated subcategory of $DM_{R}$ containing $\chow_{R}$ such that $\hw_{\chow,gm}=\chow_{R}$.

An important observation is that if %Note here, if 
 $\cha(k)=0$ or $\cha(k)=p \in S$, then $\wchose$ and $\wchosw$ is also generated by the motives of {\bf smooth projective} varieties (see Theorem 2.1.2(2) of \cite{bokum} and the proposition above). Moreover, %under such conditions the case 
  in these cases  %$DM^{eff}_{R}(k)$ and $\mgl-\modd_{R}(k)$ differ from other motivic categories in that we can describe the heart of $w^{eff}_{Smooth}$ ''geometrically''; this  is just the ''big'' category of Chow motives. 
  $DM_{R}(k)$ and $\mgl-\modd_{R}(k)$ differ from other motivic categories in that we can describe the hearts of $w_{Smooth}$ ''geometrically''; these are just the corresponding  ''big'' categories of Chow motives.  Respectively, in this case Proposition \ref{spobc}(2) gives a quite explicit calculation of unramified cohomology.  Moreover, in these case 
$w^{eff}_{Smooth}$ and $w_{Smooth}$ restrict to the corresponding subcategories  $DM^{eff}_{gm, R}(k)\subset  DM_{gm, R}(k)$ of compact objects (which  equal the %is % equivalent to the %category generated by
 smallest strict triangulated subcategories of $DM^{eff}_{R}(k)\subset DM_{R}(k)$  that contain the corresponding Chow motives), %$\chow_{R}^{eff}$), 
  and $Hw_{\chow,gm}$ is equivalent to the corresponding additive category of Chow motives $\chow_{R}$.
These types of  Chow weight structures originate from \S6.5 of \cite{bws} and \cite{bzp} (cf. also Remark 3.1.4 and Proposition 3.2.6 of \cite{bwcomp}).

Note also that one of the obstructions for inclusion $\chow_{R} \subset Hw_{\chow}$ is the non-triviality of the Hopf map $\eta$ (see \S6.2 of \cite{Mor}). %thus,
  It is closely related to orientability, cf. \cite{bchow}.

\item %Now  take
 Take  $S=\{2\}$, and recall the decomposition $SH(k)[1/2] = SH(k)^{+} \times SH(k)^{-}$ (induced by the symmetry involution of $\mathbb{P}^{1} \wedge \mathbb{P}^{1}$; see \S6 of \cite{lev}). 
%Note here
 It easily follows that we can also add $SH(k)^{+}$ to the examples listed in \S\ref{sothermot}.
\item Summarizing the above, we obtain that all the results of the previous sections can be applied directly to %any of the  each
 every  motivic category mentioned in \S\ref{vgabc}. %above.
\end{enumerate}
\end{rema}

\subsection{Supplementary remarks and questions}\label{opcacon} %{A few remarks and possible developments}

Possibly the matters mentioned below will be studied in consequent papers.

\begin{rema}
\begin{enumerate}
\item It would be interesting to generalize our weight and $t$-structure definitions and results to  relative motives over general base schemes. We also plan to study the weight-exactness for various  functors between  motivic categories.% $S$.

\item In  \cite{bchow} a $t$-structure generated by objects of the form $Th_{X}(\xi)$, where $X \in SmProj/k$ and $\xi \in K(X)$ was considered. Probably, this $t-$structure is dual (in some sense) to our $t_{\smo}^{eff}$.

\item In the papers \cite{bsoscwh} and  \cite{bonkum} %were introduced and studied in detail the 
 certain Chow-weight homology functors  from the categories $DM^{eff}_{gm}(k)$ and $DM^{eff}(k)$ were studied in detail. Using the results of this article (especially \S\ref{etfiloth}), 
one can generalize this theory to any motivic category. The most interesting case is $SH(S)$ (for a ''reasonable'' base scheme $S$).

\item For which objects $S \in \underline{Ht}^{eff}_{hom}$ is our filtration finite (i.e., there exists $N$ with $F^{N}(S)=S$)? Is it exhaustive? %See Theorem 15 of \cite{bac} and the discussion preceding his theorem for  related questions.
\end{enumerate}

\end{rema}

\end{document}